\documentclass[11pt,a4paper]{amsart}[2010]
\usepackage{tikz}
\usepackage{enumitem}
\usepackage{amscd}
\usepackage{mathabx}
\usepackage{comment}
\newcounter{TmpEnumi}
\title[Mean cohomological independence dimension]{Radius
 of comparison and mean cohomological independence dimension}

\author{Ilan Hirshberg
\and
N. Christopher Phillips}

\address{Department of Mathematics, Ben Gurion University of the Negev,
\phantom{----------------}\linebreak\text{}\hspace{3.5mm}
P.O.B. 653, Be'er Sheva 84105, Israel}
\email{ilan@math.bgu.ac.il}

\address{Department of Mathematics, University  of Oregon,
\phantom{---------------------------------}\linebreak\text{}\hspace{3.5mm}
Eugene OR 97403-1222, USA.}

\thanks{This research was supported by Israel Science Foundation
  grant 476/16 and the Simons Foundation Collaboration Grant
  for Mathematicians \#587103.}

\date{27~September 2020}

\usepackage{amsmath}
\usepackage{amssymb}
\usepackage{amsthm}
\usepackage[all]{xypic}

\allowdisplaybreaks

\theoremstyle{plain}

\newtheorem{Thm}{Theorem}[section]
\newtheorem{Cor}[Thm]{Corollary}
\newtheorem{Lemma}[Thm]{Lemma}

\newtheorem{Prop}[Thm]{Proposition}

\theoremstyle{definition}
\newtheorem{Def}[Thm]{Definition}
\newtheorem{Notation}[Thm]{Notation}

\newtheorem{Exl}[Thm]{Example}
\newtheorem{Rmk}[Thm]{Remark}
\newtheorem{Question}[Thm]{Question}

\newcommand{\andeqn}{\qquad {\mbox{and}} \qquad}

\newcommand{\R}{{\mathbb{R}}}
\newcommand{\N}{{\mathbb{N}}}
\newcommand{\Z}{{\mathbb{Z}}}

\newcommand{\Q}{{\mathbb{Q}}}

\newcommand{\cN}{{\mathcal{N}}}
\newcommand{\cS}{{\mathcal{S}}}
\newcommand{\cU}{{\mathcal{U}}}
\newcommand{\cV}{{\mathcal{V}}}
\newcommand{\cW}{{\mathcal{W}}}
\newcommand{\calD}{{\mathcal{D}}}

\newcommand{\ch}{\widecheck{H}}

\newcommand{\crd}{\mathrm{Card}}
\newcommand{\ord}{\mathrm{ord}}

\newcommand{\aut}{\mathrm{Aut}}
\newcommand{\supp}{\mathrm{supp}}

\newcommand{\eps}{\varepsilon}
\numberwithin{equation}{section}

\newcommand{\rc}{\mathrm{rc}}
\newcommand{\id}{\mathrm{id}}

\newcommand{\mdim}{\mathrm{mdim}}
\newcommand{\mcid}{\mathrm{mcid}}
\newcommand{\mcidgs}{\mathrm{smcid}}
\newcommand{\Ch}{\mathrm{Ch}}

\newcommand{\td}{d}

\newcommand{\rank}{\mathrm{rank}}
\newcommand{\Inv}{\mathrm{Inv}}

\newcommand{\af}{\alpha}

\newcommand{\dt}{\delta}
\newcommand{\ep}{\varepsilon}

\newcommand{\et}{\eta}
\newcommand{\io}{\iota}

\newcommand{\ld}{\lambda}
\newcommand{\sm}{\sigma}

\newcommand{\rh}{\rho}

\renewcommand{\S}{\subset}
\newcommand{\I}{\infty}

\newcommand{\Act}{T}

\begin{document}

\begin{abstract}
We introduce a notion of mean cohomological independence dimension
for actions of discrete amenable groups on compact metrizable spaces,
as a variant of mean dimension,
and use it to obtain lower bounds for the radius of comparison
of the associated crossed product $C^*$-algebras.
Our general theory,
gives the following for the minimal subshifts constructed by Dou
in 2017.
For any countable amenable group~$G$ and any polyhedron~$Z$,
Dou's subshift $T$ of $Z^G$ with density parameter $\rh$
satisfies
\[
\rc (C (X) \rtimes_{\Act} G)
 > \frac{1}{2} \mdim (T) \left( 1 - \frac{1 - \rh}{\rh} \right) - 2.
\]
If $k = \dim (Z)$ is even and $\ch^k (Z; \Q) \neq 0$,
then
\[
\rc (C (X) \rtimes_{\Act} G) > \frac{1}{2} \mdim (T) - 1,
\]
regardless of what $\rh$ is.
\end{abstract}

\maketitle

The notion of mean dimension was outlined by Gromov in \cite{Gromov},
and later fleshed out in a paper
of Lindenstrauss and Weiss \cite{LinWeiss2000}.
The general philosophy outlined in Gromov's paper
was that given an invariant $\Inv(X)$ for spaces $X$,
one can try to define a dynamical variant $\Inv(X;G)$
for actions of groups on $X$,
which should, as a test case, for the full shift
roughly satisfy $\Inv(X^G;G) = \Inv(X)$.
Of course, various restrictions may be placed on the spaces,
on the groups, or on the actions.
Entropy, for instance,
can be thought of as a dynamical way to count cardinality.
The mean dimension $\mdim (X, G)$
is a dynamical variant of covering dimension.
For actions of~$\Z$, see Definition~2.6 of \cite{LinWeiss2000};
for amenable groups, see the remarks after this definition
and the discussion of this case
in the introduction to \cite{LinWeiss2000}.
One of the motivating applications was to show that
not every dynamical system of the form $(X, \Z)$ can be embedded
into the full shift on $[0,1]^{\Z}$.
If $T$ denotes the action of $G$ on $X$,
we sometimes write $\mdim(T)$ in place
of $\mdim (X,G)$.

The notion of radius of comparison for $C^*$-algebras
was introduced by Toms in \cite{Toms-rc},
as a way to systematize the counterexamples
to the Elliott program he constructed in \cite{Toms-counterexample},
based on techniques introduced first by Villadsen in \cite{Villadsen}.
Let $A$ be a unital stably finite $C^*$-algebra.
Let $\tau$ be a tracial state on $A$.
By slight abuse of notation,
we also use $\tau$ to denote the induced trace on $M_{\infty}(A)$.
For a positive element $a \in M_n(A)$,
we set $d_{\tau}(a) = \lim_{n \to \infty} \tau(a^{1/n})$.
For $r>0$,
we say that $A$ has \emph{$r$-comparison}
if for any two positive elements $a,b \in M_{\infty}(A)$,
if $d_{\tau}(a) + r < d_{\tau}(b)$ for all tracial states $\tau$ on $A$
then $a \precsim b$ ($a$ is Cuntz-subequivalent to $b$).
(In general, one should use quasitraces here,
but the $C^*$-algebras in this paper will be nuclear,
so that all quasitraces are tracial states by \cite{Haagerup}.)
The \emph{radius of comparison} of $A$ is the infimum of all $r>0$
such that $A$ has $r$-comparison.
Toms' counterexample is of a simple AH algebra
which has positive radius of comparison but otherwise
has the same Elliott invariant as an AI algebra
(which has zero radius of comparison).
Recent major advances in the study of classification theory
for nuclear
$C^*$-algebras (\cite{elliott-gong-lin-niu,tikuisis-white-winter}),
building on decades of work by many authors,
show that simple nuclear unital $C^*$-algebras
satisfying the Universal Coefficient Theorem are classified via
the Elliott invariant provided they have finite nuclear dimension.
Conjecturally, this corresponds to the case of
zero radius of comparison;
this has been proved when the tracial state space of $A$
is a Bauer simplex
whose extreme boundary has finite covering
dimension (\cite{kirchberg-rordam,sato,toms-white-winter}).

The connection to dynamical systems was broached by Giol and Kerr
in \cite{giol-kerr},
where they constructed examples of minimal homeomorphisms
whose crossed products have positive radius of comparison.
The examples in the paper of Giol and Kerr
have positive mean dimension.
That the spaces themselves had to be infinite dimensional
follows from the fact that for minimal homeomorphisms
of finite dimensional spaces,
the crossed product has finite nuclear dimension
and hence has zero radius of comparison (\cite{TomsWinter2013};
see also \cite{HirshbergWinterZacharias}).
This suggested a connection
between mean dimension and radius of comparison,
two notions which came about independently and in different contexts.
It has been conjectured by the second named author and Toms
that for minimal systems,
the radius of comparison should be half the mean dimension.
The second named author showed in \cite{Phillips2016}
that for minimal actions $T$ of the integers,
the radius of comparison of the crossed product is bounded above
by $1 + 36 \mdim (T)$.
Elliott and Niu showed in \cite{ElliottNiu2017}
that for minimal actions of the integers,
mean dimension zero implies zero radius of comparison.
Recently, Niu (\cite{Niu2019Z,Niu2019Zd})
improved those results
and showed that for free and minimal actions of $\Z^d$,
the radius of comparison of the crossed product
is at most half the mean dimension.
Those results mark very significant progress on this problem,
but they all involve bounding the radius of comparison from above.
For lower bounds,
the only results we are aware of to date
are for the examples constructed in the paper of Giol and Kerr.

The goal of this paper is to establish lower bounds
(Theorem~\ref{T_0826_Main} and Theorem~\ref{T_0907_gs}).
In the case of commutative $C^*$-algebras,
lower bounds for the radius of comparison of $C(X)$
were obtained in \cite{ElliottNiu2013}
in terms of rational cohomological dimension
rather than covering dimension.
For $t \in \R$,
we denote by $\lfloor t \rfloor$ the greatest integer $n$
such that $n \leq t$.
When covering dimension and rational cohomological dimension coincide,
the radius of comparison of $C(X)$ is $\lfloor \dim(X)/2 \rfloor - 1$
or $\lfloor \dim(X)/2 \rfloor - 2$;
it is not known whether the latter can occur.
We refer to \cite{DranishnikovSurvey}
for a survey of cohomological dimension.

The work of Elliott and Niu \cite{ElliottNiu2013} suggests that,
in order to obtain lower bounds in the dynamical context,
rather than using the Lindenstrauss-Weiss notion of mean dimension,
which is based on covering dimension,
we might look for a notion of ``mean cohomological dimension''.
Recall that a compact metrizable space $X$
has rational cohomological dimension $d$
if $d$ is the least integer such that for any $k > d$
and for any closed subset $Y \subset X$,
we have $\ch^k(X, Y; \Q) = 0$
(relative \v{C}ech cohomology with rational coefficients).

Instead of relative cohomology,
for our purposes it turns out to be better to work
with cohomology of subspaces: given a space $X$,
for any $k$ we can ask whether there exists a closed subset $Y$ of $X$
with non-vanishing $k$-th rational
cohomology.
We could define a notion of the dimension of~$X$ as
the supremum of all $k \in \N$
such that this holds.
Such a notion does not quite coincide with covering dimension
for CW complexes.
(For instance, the dimension of the three dimensional ball
would be $2$ rather than $3$.)
However, in the context of mean dimension,
it sometimes does not matter
if the dimension it is based on is off by a constant.
For technical reasons, we actually consider only even integers~$k$.
This is related to the fact that we work with complex vector bundles
and Chern classes; more philosophically,
it reflects the fact that the radius of comparison should be thought of
as a sort of complex dimension rather than real dimension,
which explains the factor of $1/2$
which appears when comparing it to mean dimension.

In fact, it turns out to be more useful to view this as
a sequence of invariants:
for any $k$,
we could ask whether there exists a subspace
with non-vanishing $k$-th rational cohomology.
We think of cohomology classes $\eta_1,\eta_2,\ldots,\eta_n$
as being ``independent'' if their cup product is nonzero.
Roughly speaking,
given an action of an amenable group $G$ on $X$,
and given a cohomology class $\eta$ of a subspace,
for any finite set $F$ of $G$,
we can find the largest subset $F_0$ such that
the iterates of $\eta$ under $F_0$ are independent in this sense,
and then measure the upper density of such sets in a F{\o}lner sequence.
This is used as a basis for
our notion of \emph{mean cohomological independence dimension}
(Definition~\ref{Def_mcid_Rev}).
For full shifts on a CW complex $Z$,
under a mild condition on the group,
our dynamical invariant recovers the dimension of $Z$,
thereby meeting the rule of thumb
suggested in Gromov's paper \cite{Gromov}.
There is a related but somewhat different notion of
mean homological dimension in Section 2.6.3 of \cite{Gromov}.
It applies specifically to subshifts,
and is used there to get lower bounds on the mean dimension.
The reader may also find some analogy between the connection
of mean dimension
with our notion of independence dimension
and the connection of entropy with combinatorial independence
which was studied by Kerr and Li (\cite{KerrLi2007}),
although we do not use it in any way in this paper.

In Section~\ref{Sec_0922_mcid}, we give some preliminaries
and define the mean cohomological independence dimension.
Section~\ref{Sec_0922_SbShift}
contains estimates of the mean cohomological independence dimension
of shifts and certain kinds of subshifts.
Section~\ref{Sec_0908_LB} contains the main theorem,
giving a lower bound on $\rc (C (X) \rtimes_{\Act} G)$
in terms of mean cohomological independence dimension.
When applied to a minimal subshift $T$ of the shift on $Z^G$,
as constructed in \cite{Dou2017},
using a polyhedron~$Z$ and density parameter~$\rh$,
this result implies (Corollary~\ref{C_0926_Shift_Gen})
\[
\rc (C (X) \rtimes_{\Act} G)
 > \frac{1}{2} \mdim (T) \left( 1 - \frac{1 - \rh}{\rh} \right) - 2.
\]
Up to an additive constant,
this estimate is close to the conjectured value $\frac{1}{2} \mdim (T)$
when $\rh$ is close to~$1$,
but is useless if $\rh \leq \frac{1}{2}$.
In Section \ref{Sec_0908_Symm},
we introduce a variant of our definition, which
we call \emph{symmetric mean cohomological independence dimension}.
This involves a stronger independence condition,
which allows us to obtain
improved bounds for certain dynamical systems,
such as subshifts of $(S^k)^G$ for $k$ even.
In particular, for the construction of \cite{Dou2017}
in this case, one gets (Corollary~\ref{C_0907_Shift_Sk})
\[
\rc (C (X) \rtimes_{\Act} G) > \frac{1}{2} \mdim (T) - 1,
\]
regardless of the value of~$\rh$,
which is useful whenever $\rh > \frac{2}{k}$.
In Section~\ref{Sec_0922_CR}, we state some open problems.

\section{Mean cohomological independence dimension}\label{Sec_0922_mcid}

We begin by fixing some notation.
Throughout this paper,
$X$ is a compact metrizable space,
$G$ is a countable amenable discrete group,
and $\Act$ is an action of $G$ on~$X$.
When needed,
we let $\alpha \colon G \to \aut(C(X))$ be the corresponding action
of $G$ on $C (X)$,
given by $\alpha_g (f) (x) = f (\Act_g^{-1} (x))$.
We usually write the crossed product
$C (X) \rtimes_{\alpha} G$ as $C (X) \rtimes_{\Act} G$.
(Since $G$ is amenable,
the full and reduced crossed products are the same.)

\begin{Def}\label{D_0814_CoverY}
Let $X$ be a compact metrizable space,
and let $Y \S X$ be closed.
A {\emph{finite open cover of $Y$ in~$X$}}
is a finite collection of
nonempty open subsets of~$X$ whose union contains~$Y$.
We often omit mention of~$X$ when it is understood.
We denote by $\cN (\cU)$ the nerve of~$\cU$.
\end{Def}

To emphasize: a finite open cover of $Y$
consists of subsets of~$X$ which are open in~$X$,
not of open subsets of~$Y$.
For this and the following definitions
(but not for some of the lemmas),
there is no reason not to use arbitrary topological spaces~$X$
and arbitrary subsets~$Y$.

We exclude $\varnothing$ from covers to avoid later
improperly claiming that
$\cU \cup \{ \varnothing \} = \cU$.

The less convenient alternative is to work with
finite open covers of various closed sets~$Y$ in the traditional sense.
The outcome will be the same;
see Lemma~\ref{L_0815_CoverY} and Lemma~\ref{L_0820_DimSame} below.

We now give definitions which are standard for open covers,
slightly modified for our present situation.

\begin{Def}\label{D_0815_Join}
Let $X$ be a compact metrizable space,
and let $\cU_1$ and $\cU_2$ be collections of nonempty open sets in~$X$.
Then their {\emph{join}} is
\[
\cU_1 \vee \cU_2
 = \bigl\{ U_1 \cap U_2 \mid {\mbox{$U_1 \in \cU_1$,
     $U_2 \in \cU_2$, and $U_1 \cap U_2 \neq \varnothing$}} \bigr\}.
\]
\end{Def}

By iteration, we get the join
of any finite set of collections of open sets.
If $\cU_1, \cU_2, \ldots, \cU_n$ are
finite open covers in~$X$
of closed subsets $Y_1, Y_2, \ldots, Y_n \S X$,
then $\cU_1 \vee \cU_2 \vee \cdots \vee \cU_n$
is a finite open cover of $Y_1 \cap Y_2 \cap \cdots \cap Y_n$ in~$X$.

\begin{Def}\label{D_0808_Rfn}
Let $X$ be a compact metrizable space,
let $Y \S X$ be closed,
and let $\cU$ and $\cV$ be finite open covers of~$Y$.
Then $\cV$ {\emph{refines $\cU$}} (as a cover of~$Y$;
written $\cV \prec_Y \cU$)
if
% $\bigcup_{U \in \cU} U = \bigcup_{V \in \cV} V$ and
for every $V \in \cV$ there is $U \in \cU$
such that $V \subset U$.
\end{Def}

Formally, the only role that $Y$ plays is that we are
restricting to finite open covers of~$Y$.

\begin{Def}\label{D_0807_Covers}
Let $X$ be a compact metrizable space,
let $Y \S X$ be closed,
and let $\cU$ be a finite open cover of~$Y$ in~$X$.
The {\emph{order of $\cU$}} is
\[
\ord (\cU)
% = \max_{x \in \bigcup \cU}
 = \max_{x \in X}
      \crd \bigl( \{U \in \cU \mid x \in U\} \bigr) - 1.
\]
We denote by $\calD_Y (\cU)$
the least order of any refinement of $\cU$ among
finite open covers of~$Y$ in~$X$.
\end{Def}

In the situation of Definition~\ref{D_0807_Covers},
the order of $\cU$ is the dimension of $\cN (\cU)$.

While the subset~$Y$ is formally irrelevant in
Definition~\ref{D_0808_Rfn},
the quantity $\calD_Y (\cU)$ depends strongly on~$Y$.

\begin{Notation}\label{N_0820_Int}
Let $X$ be a compact metrizable space,
let $Y \S X$ be closed,
and let $\cU$ be a finite open cover of~$Y$ in~$X$.
We set
\[
\cU \cap Y
 = \bigl\{ U \cap Y \mid
   {\mbox{$U \in \cU$ and $U \cap Y \neq \varnothing$}} \bigr\},
\]
which is a finite open cover of~$Y$,
regarded as a topological space in its own right.
\end{Notation}

\begin{Lemma}\label{L_0815_CoverY}
Let $X$ be a compact metrizable space,
let $Y \S X$ be closed,
and let $\cV$ be a finite open cover of~$Y$,
regarded as a topological space in its own right.
Then there is a finite open cover $\cW$ of~$Y$ in~$X$
such that $\ord ( \cW ) \leq \ord ( \cV )$
and $\cW \cap Y$ refines $\cV$
% $\bigl\{ W \cap Y \mid W \in \cW \bigr\}$ refines $\cV$
in the conventional sense for open covers of~$Y$ in~$Y$.
\end{Lemma}

\begin{proof}
Write $\cV = \{ V_1, V_2, \ldots, V_n \}$
with $V_1, V_2, \ldots, V_n$ distinct.
Choose open subsets $U_1, U_2, \ldots, U_n \S X$
such that $V_j = U_j \cap Y$ for $j = 1, 2, \ldots, n$.
Choose open subsets $U_{j, l} \S X$ for $l = 1, 2, \ldots$ with
\[
U_{j, 1} \S {\overline{U_{j, 1}}}
 \S U_{j, 2} \S {\overline{U_{j, 2}}} \S \cdots \S U_j
\andeqn
\bigcup_{l = 1}^{\I} U_{j, l} = U_j.
\]
Then
\[
Y \subset \bigcup_{l = 1}^{\I}
 \left( \bigcup_{j = 1}^{n} U_{j, l}  \right),
\]
so, by compactness,
there is $l_0 \in \{ 1, 2, \ldots \}$ such that
$Y \S \bigcup_{j = 1}^{n} U_{j, l_0}$.

Choose open subsets $Z_1, Z_2, \ldots \S X$ such that
\[
Z_1 \supset {\overline{Z_2}}
 \supset Z_2 \supset {\overline{Z_3}} \supset \cdots \supset Y
\andeqn
\bigcap_{m = 1}^{\I} Z_m = Y.
\]
Let $\cS$ be the set of all subsets $J \S \{ 1, 2, \ldots, n \}$
such that $\bigcap_{j \in J} V_j = \varnothing$.
For $J \in \cS$, we have
\[
\bigcap_{m = 1}^{\I} \left( \bigcap_{j \in J}
    \bigl( {\overline{U_{j, l_0} }} \cap {\overline{Z_m}} \bigr) \right)
 = \varnothing.
\]
Therefore there exists $m_J \in \{ 1, 2, \ldots \}$
such that
\[
\bigcap_{j \in J}
    \bigl( {\overline{U_{j, l_0} }} \cap {\overline{Z_{m_J} }} \bigr)
 = \varnothing.
\]
Define $m = \max_{J \in \cS} m_J$.
For $j = 1, 2, \ldots, n$ define $W_j = U_{j, l_0} \cap Z_{m}$,
and set $\cW = \{ W_1, W_2, \ldots, W_n \}$.
One easily checks that $\cW$ satisfies the conclusion of the lemma.
\end{proof}

\begin{Lemma}\label{L_0820_DimSame}
Let $X$ be a compact metrizable space,
let $Y \S X$ be closed,
and let $\cU$ be a finite open cover of~$Y$ in~$X$.
Then $\calD_Y (\cU) = \calD_Y ( \cU \cap Y )$.
\end{Lemma}

The expression $\cU \cap Y$ is as in Notation~\ref{N_0820_Int},
and $\calD_Y ( \cU \cap Y )$ is what is usually called
$\calD ( \cU \cap Y )$, taken among finite open covers of~$Y$
as a topological space in its own right.

\begin{proof}[Proof of Lemma~\ref{L_0820_DimSame}]
We first claim that $\calD_Y (\cU) \leq \calD_Y ( \cU \cap Y )$.
Choose a finite open cover $\cV$ of $Y$ in~$Y$
such that $\cV \prec_Y \cU \cap Y$
and $\ord (\cV) = \calD_Y ( \cU \cap Y )$.
Use Lemma~\ref{L_0815_CoverY} to choose
a finite open cover $\cW_0$ of $Y$ in~$X$
such that $\cW_0 \cap Y \prec_Y \cV$ and $\ord (\cW_0) \leq \ord (\cV)$.
For each $W \in \cW_0$ there is $U_W \in \cU$
such that $W \cap Y \S U_W \cap Y$.
Set $\cW = \bigl\{ W \cap U_W \mid W \in \cW_0 \bigr\}$.
Clearly $\cW$ covers $Y$ in~$X$,
$\cW \precsim_Y \cU$, and $\ord (\cW) \leq \ord (\cW_0)$.
So
\[
\calD_Y ( \cU )
 \leq \ord (\cW)
 \leq \ord (\cW_0)
 \leq \ord (\cV)
 = \calD_Y ( \cU \cap Y ),
\]
proving the claim.

For the reverse inequality,
choose a finite open cover $\cW$ of $Y$ in~$X$
such that $\cW \prec_Y \cU$ and $\ord (\cW) = \calD_Y ( \cU )$.
Then
\[
\calD_Y ( \cU \cap Y )
 \leq \ord ( \cW \cap Y)
 \leq \ord ( \cW )
 = \calD_Y (\cU),
\]
as desired.
\end{proof}

In particular, the covering dimension $\dim (Y)$
can be calculated using open covers of $Y$ in~$X$
instead of conventional open covers of~$Y$.

\begin{Def}\label{D_0814_CechFromCov}
Let $X$ be a compact metrizable space,
let $Y \S X$ be closed, and let $\cU$
be a finite open cover of $Y$ in~$X$.
Let $R$ be a commutative unital ring.
For $k = 0, 1, 2, \ldots$,
we denote by $\ch^{k} (Y; \cU; R)$ the set of those elements
of the \v{C}ech cohomology group $\ch^{k} (Y; R)$
which can be represented by cocycles arising from
% the open cover of~$Y$,
% in the conventional sense,
% given by $\bigl\{ U \cap Y \mid U \in \cU \bigr\}$.
$\cU \cap Y$ (as in Notation~\ref{N_0820_Int}).
\end{Def}

Suppose $Y_1$ and $Y_2$ are two closed subsets of $X$.
Suppose $\eta_1 \in \ch^k(Y_1 ; R)$ and $\eta_2 \in \ch^m(Y_2;R)$.
Though we cannot define the cup product of these two elements,
as they belong to different groups,
we can restrict them to the intersection and consider the cup product
\[
\eta_1 |_{Y_1 \cap Y_2} \smile \eta_2 |_{Y_1 \cap Y_2}
  \in \ch^{k + m} (Y_1 \cap Y_2; R).
\]

\begin{Lemma}\label{L_0815_CP_cover}
Let $X$ be a compact metrizable space,
let $Y_1, Y_2 \S X$ be closed,
let $R$ be a commutative unital ring,
and for $j = 1, 2$ let $\cU_j$ be a finite open cover of $Y_j$ in~$X$,
let $m_j \in \{ 0, 1, 2, \ldots \}$,
and let $\eta_j \in \ch^{m_j} ( Y_j; \cU_j; R )$.
Then
\[
\eta_1 |_{Y_1 \cap Y_2} \smile \eta_2 |_{Y_1 \cap Y_2}
  \in \ch^{m_1 + m_2} (Y_1 \cap Y_2; \, \cU_1 \vee \cU_2; \, R).
\]
\end{Lemma}

\begin{proof}
% This follows immediately from the construction of the cup product
% in \v{C}ech cohomology.
Set $Y = Y_1 \cap Y_2$,
and set $\cV_j = \cU_j \cap Y$,
a finite open cover of~$Y$.
It suffices to prove that if $\et_j \in \ch^{m_j} ( Y; \cV_j; R )$
for $j = 1, 2$,
then
$\eta_1 \smile \eta_2
  \in \ch^{m_1 + m_2} (Y; \, \cV_1 \vee \cV_2; \, R)$.
Since $\cV_1 \vee \cV_2$ refines both $\cV_1$ and $\cV_2$, we have
$\eta_j \in \ch^{m_j} (Y; \, \cV_1 \vee \cV_2; \, R)$
for $j=1,2$.
This implies that
there exist a map $h \colon Y \to \cN(\cV_1 \vee \cV_2)$
and elements $\xi_j \in \ch^{m_j} (\cN (\cV_1 \vee \cV_2) ; \, R)$
such that $\eta_j = h^*(\xi_j)$ for $j=1,2$.
Thus,
$\xi_1 \smile \xi_2 \in \ch^{m_1+m_2}
  (\cN (\cV_1 \vee \cV_2) ; \, R)$,
and because
$\eta_1 \smile \eta_2 = h^* ( \xi_1 \smile \xi_2 )$,
we have
$\eta_1 \smile \eta_2
 \in \ch^{m_1 + m_2} (Y; \, \cV_1 \vee \cV_2; \, R)$,
as claimed.
\end{proof}

If $F$ is a (finite) set,
we denote by $\crd(F)$ the cardinality of $F$.

\begin{Def}\label{D_0807_AppInv}
If $G$ is a group, $F \subset G$ is a nonempty finite subset,
$G_0 \subset G$, and $\delta>0$,
we say that $F$ is $(G_0,\delta)$-invariant
if $\crd(gF \cap F) > (1-\delta)\crd(F)$ for all $g \in G_0$.
\end{Def}

By convention, $(G_0, \delta)$-invariant subsets are nonempty.

\begin{Def}\label{Def_mcid_Rev}
Let $X$ be a compact metrizable space,
let $G$ be a countable amenable group,
and let $\Act$ be an action of $G$ on $X$.
Let $R$ be a commutative unital ring.
For any \emph{even} integer $k$,
we define $\mcid_k (\Act; R)$
to be the largest $d \in [0, \I)$
such that the following happens.

For every $\ep > 0$ there are a closed subset $Y \S X$,
a finite open cover $\cU$ of $Y$ in~$X$
such that $\calD_Y (\cU) \in \{ k, k + 1 \}$,
and $\et \in \ch^k (Y; \cU; R)$
(Definition~\ref{D_0814_CechFromCov}: \v{C}ech classes using $\cU$)
such that for every finite subset $G_0 \S G$ and every $\dt > 0$
there are a $(G_0, \dt)$-invariant nonempty finite set $F \S G$
and a subset $F_0 \S F$ for which the following happen:
\begin{enumerate}
% [label=$\mathrm{(\arabic*)}$]
%
\item\label{Item_D_0808_mcid_Prod}
The cup product of $\Act_g^{*} (\et)$ over all $g \in F_0$,
which makes sense as an element of
$\ch^{k \cdot \crd (F_0)}
 \bigl( \bigcap_{g \in F_0} T_g^{-1} (Y); R \bigr)$,
is nonzero.
\item\label{Item_D_0808_mcid_Large}
$\dfrac{k \cdot \crd (F_0)}{\crd (F)} > d - \ep$.
\end{enumerate}

We then say that $\Act$ has \emph{mean $k$-th
cohomological independence dimension $d$ with coefficients in $R$}.

We define the \emph{mean cohomological independence dimension}
$\mcid (\Act; R)$
to be the supremum of $\mcid_k (\Act; R)$
over all $k \in \{ 0, 2, 4, 6, \ldots \}$.
\end{Def}

\begin{Rmk}\label{R_0820_AtMostk}
The empty cup product is taken to be~$1$.
It is then easy to check that the set of $d$
singled out by the second paragraph of the definition
is an interval of the form $[0, r]$ for some $r \in [0, k]$.
In particular, $\mcid_k (\Act; R) \leq k$.
\end{Rmk}

\begin{Rmk}\label{R_0808_OnlyUseQ}
We defined mean cohomological independence dimension
with coefficients in an arbitrary commutative unital ring~$R$,
but in this paper we only use the case $R = \Q$.
One could think of various generalizations
in which one replaces \v{C}ech cohomology
by more general sheaf cohomology.
However we do not have any use for that here.

We took $k$ to be even for technical reasons
which will become apparent later.
Essentially, this is as we intend to work with complex vector bundles,
and in some sense our notion of dimension
should be thought of as complex dimension.
We could have made the definition without this restriction,
noting that, as the cup product would no longer be commutative,
we would have to make an arbitrary choice for the order
in which the product is taken.
We do not know whether a version which allows odd values of~$k$
is significantly different.
\end{Rmk}

\begin{Rmk}\label{Rmk:mcid for invariant subsets}
It is immediate from the Definition \ref{Def_mcid_Rev}
that if $\Act$ is an action of $G$ on a compact metrizable space $X$
and $Y$ is a closed invariant subset,
then $\mcid_k (\Act|_Y ; R) \leq \mcid_k (\Act ; R)$ for any $k$,
and $\mcid (\Act|_Y ; R) \leq \mcid (\Act ; R)$.
\end{Rmk}

\section{Mean cohomological independence dimension of
  subshifts}\label{Sec_0922_SbShift}

In this section, we give estimates on $\mcid (\Act ; R)$
when $\Act$ is a shift,
or a subshift of the type considered in
\cite{giol-kerr}, \cite{Krieger2009}, and~\cite{Dou2017}.
We start with a general result:
$\mcid_k (\Act; R) \leq \mdim (\Act)$.

\begin{Prop}\label{Prop_mcid_bounded_by_mdim}
Let $X$ be a compact metrizable space,
let $G$ be a countable amenable group,
let $\Act$ be an action of $G$ on $X$,
and let $R$ be a commutative unital ring.
Then for any even natural integer~$k$,
we have $\mcid_k (\Act; R) \leq \mdim (\Act)$.
\end{Prop}

\begin{proof}
Without loss of generality $\mdim (\Act) < \I$.

Fix an even integer~$k$, and let $\eps > 0$.
Pick a compact subset $Y$ of $X$,
a finite open cover $\cU$ of $Y$ in~$X$ such that
% $k = 2 \left\lfloor \frac{\calD (\cU)}{2} \right\rfloor$,
% $\calD (\cU) = k$ or $\calD (\cU) = k + 1$,
$\calD_Y (\cU) \in \{ k, \, k + 1 \}$,
and an element $\eta \in \ch^k(Y;\cU;R)$ such that
for any finite subset $G_0 \subset G$ and any $\delta>0$
there exists a finite $(G_0,\delta)$-invariant subset $F \S G$
and a subset $F_0 \subset F$ for which:
\begin{enumerate}
\item\label{0816_mcid_mdim_1}
$\displaystyle \underset{g \in F_0}{\smile}
       \Act_{g}^* (\eta) |_{\bigcap_{h \in F_0} \Act_{h}^{-1} (Y)}
\neq 0$.
\item\label{0816_mcid_mdim_2}
$\displaystyle \frac{k \cdot  \crd (F_0)}{\crd (F)}
  > \mcid_k (\Act; R) - \eps$.
\end{enumerate}

Recall that $\cU$ may not be a cover of $X$.
To remedy that,
set $\cU' = \cU \cup \{X \smallsetminus Y\}$.
By the definition of $\mdim (\Act)$,
there are a finite subset $G_0 \subset G$ and $\delta > 0$
such that for every $(G_0, \delta)$-invariant subset $F \S G$
we have
\[
\frac{ \calD_X
  \bigl( \bigvee_{g \in F} \Act_{g^{-1}} (\cU') \bigr)}{\crd (F)}
 < \mdim (\Act) + \eps.
\]
Choose sets $F$ and $F_0$ as above to go with this choice of~$G_0$.
Set
\[
Z = \bigcap_{g \in F} \Act_{g^{-1}} (Y)
\andeqn
Z_0 = \bigcap_{g \in F_0} \Act_{g^{-1}} (Y).
\]
Further define
\[
\cW' = \bigvee_{g \in F} \Act_{g^{-1}} (\cU'),
\qquad
\cW_0 = \bigvee_{g \in F_0} \Act_{g^{-1}} (\cU),
\andeqn
\cW_0' = \bigvee_{g \in F_0} \Act_{g^{-1}} (\cU').
\]
(We don't need the cover that would logically be called~$\cW$.)
Then $\cW'$ and $\cW_0'$ are open covers of~$X$
and $\cW_0$ is an open cover of $Y_0$ in~$X$.

We claim that, following Notation~\ref{N_0820_Int},
we have $\cW_0' \cap Z_0 = \cW_0 \cap Z_0$.
To prove the claim,
first recall that $\cW_0 \cap Z_0$ is the set of nonempty sets in
\[
\left\{ \bigcap_{g \in F_0} \bigl( T_g^{-1} (U_g) \cap Z_0 \bigr) \mid
 {\mbox{$U_g \in \cU$ for $g \in F_0$}} \right\}.
\]
In $\cW_0' \cap Z_0$,
we must also allow $U_g = X \smallsetminus Y$
for some values of $g \in F_0$,
but the additional sets gotten this way are all empty.
The claim is proved.

We have (justifications afterwards)
\[
\calD_X (\cW')
 \geq \calD_{Z_0} (\cW')
 \geq \calD_{Z_0} (\cW_0')
 = \calD_{Z_0} (\cW_0).
\]
The first inequality follows from $Z_0 \S X$,
the second from $\cW' \prec_{Z_0} \cW_0'$,
and the third from the claim above and two applications
of Lemma~\ref{L_0820_DimSame}.

Now,
\[
\underset{g \in F_0}{\smile}
 \Act_{g}^* (\eta) |_{Z_0}
 \in \ch^{k \crd (F_0)} ( Z_0; \cW_0; R ),
\]
and therefore, in particular,
%  \[
%  \ch^{k \crd(F_0)} \left(
%   \bigcap_{g \in F_0} \Act_{g^{-1}} (Y);
%      \bigvee_{g \in F_0} \Act_{g^{-1}} (\cU); R \right)
%  \neq 0.
%  \]
%
% the group on the right is nonzero.
$\ch^{k \crd (F_0)} ( Z_0; \cW_0; R ) \neq 0$.
As the elements in this cohomology group
can be realized as pullbacks of elements from nerves
of arbitrary refinements of $\cW_0 \cap Y$,
it follows from Lemma~\ref{L_0820_DimSame} that
$\calD_{Z_0} (\cW_0) \geq k \crd (F_0)$.
Therefore
\begin{equation*}
\begin{split}
\mdim (\Act) + \eps
 > \frac{ \calD_X (\cW')}{\crd (F)}
 \geq \frac{ k \crd (F_0)}{ \crd(F) }
 \geq \mcid_k (\Act; R) - \eps.
\end{split}
\end{equation*}
As $\eps > 0$ is arbitrary,
we have $\mdim (\Act) \geq \mcid_k (\Act; R)$, as required.
\end{proof}

Recall that if $G$ is a discrete group and $Z$ is a set,
then the shift action of $G$ on $Z^G$
is given by $\Act_g (x)_h = x_{g^{-1} h}$
for any $x = (x_g)_{g \in G} \in Z^G$
and all $g, h \in G$.

For the remainder of this section, we add
the assumption that $R$ is a principal ideal
domain.
This is done because the properties of cup products which we need are
often derived in the context of the K{\"u}nneth Formula.
While this requirement
is not strictly needed for the estimates in this section,
we have no present use for such a
possible generalization.

\begin{Prop}\label{Prop_mcid_for_full_shifts}
Let $Z$ be a finite CW complex,
and let $G$ be a countable amenable group.
Suppose $X = Z^G$ and let $\Act$ be the shift.
Let $R$ be a principal ideal domain.
\begin{enumerate}
\item\label{Prop_mcid_for_full_shifts_pI}
For any even $k < \dim(Z)$
we have $\mcid_k (\Act; R) = k$.
\item\label{Prop_mcid_for_full_shifts_mcidEst}
We have
\[
\mcid (\Act; R)
 \geq 2 \left\lfloor \frac{\dim(Z) - 1}{2} \right\rfloor.
\]
\item\label{Prop_mcid_for_shifts_RFin}
If furthermore $G$ has a quotient which is infinite
and residually finite, then $\mcid (\Act; R) = \dim (Z)$.
\end{enumerate}
\end{Prop}

The hypothesis in~(\ref{Prop_mcid_for_shifts_RFin})
is equivalent to saying that
$G$ has subgroups of arbitrarily large finite index.

\begin{proof}[Proof of Proposition~\ref{Prop_mcid_for_full_shifts}]
We prove~(\ref{Prop_mcid_for_full_shifts_pI}).
Let $k$ be an even integer such that $k < \dim(Z)$.
As $Z$ has a cell of dimension greater than $k$,
we can embed the sphere $S^k$ into $Z$.
To simplify notation, fix a copy of $S^k$ in~$Z$,
and simply write $S^k \subset Z$.
Let $q \colon Z^G \to Z$ be the projection onto the coordinate $g=1$.
Set $Y = q^{-1}(S^k)$.
Fix an isomorphism $\ch^k (S^k; R) \to R$.
Let $\et_0 \in \ch^k (S^k; R)$
be the element mapped to the identity of $R$
under this isomorphism.
% $\ch^k (S^k; R) \to R$.
Choose a finite open cover $\cV_0$ of $S^k$
for which $\ch^k (S^k; \cV_0; R) \to \ch^k (S^k; R)$
is an isomorphism.
Choose a finite open cover $\cV$ of $S^k$ in~$Z$
such that $\cV \cap S^k = \cV_0$.
Then $\ch^k (S^k; \cV; R) = \ch^k (S^k; \cV_0; R)$,
so $\et_0 \in \ch^k (S^k; \cV; R)$.
Set $\cU = q^{-1} (\cV)$,
which is a finite open cover of $Y$ in~$Z^G$,
and set $\eta = (q |_Y)^* (\eta_0) \in \ch^k (Y; \cV; R)$.
We have $\calD_{S^k} (\cV_0) = k$,
so $\calD_{S^k} (\cV) = k$ by Lemma~\ref{L_0820_DimSame}.
It is now easily seen that $\calD_{Y} (\cU) \leq k$,
and the reverse inequality follows from
$\ch^k (Y; \cV; R) \neq 0$,
which is a consequence of the next claim.

We claim that if $F \subset G$ is finite,
then
\[
{\displaystyle{
 \underset{g \in F}{\smile}
   \Act_{g}^*(\eta) |_{\bigcap_{h \in F} \Act_{h}^{-1} (Y)} \neq 0. }}
\]
This will imply that $\mcid_k (\Act; R) \geq k$.
Since $\mcid_k (\Act; R) \leq k$ by Remark~\ref{R_0820_AtMostk},
it follows that $\mcid_k (\Act; R) = k$.

To prove the claim,
choose any point $y_0 \in Z^{G \smallsetminus F}$.
Identify
\[
\bigcap_{g \in F} \Act_{g^{-1}} (Y)
 = (S^k)^F \times Z^{G \smallsetminus F}.
\]
Define maps
\[
t \colon (S^k)^F \to \bigcap_{g \in F} \Act_{g^{-1}} (Y)
\andeqn
p \colon \bigcap_{g \in F} \Act_{g^{-1}} (Y) \to (S^k)^F
\]
by $t (x) = (x, y_0)$ for $x \in (S^k)^F$ and $p (x, y) = x$
for $x \in (S^k)^F$ and $y \in Z^{G \smallsetminus F}$.
Then $p \circ t = \id_{(S^k)^F}$.
Set
\[
\mu
 = \prod_{g \in F} \et_0
 \in \ch^{k \crd (F)} \bigl( (S^k)^F; R \bigr).
\]
(The order in the product does not matter because $k$ is even.)
Then $\mu$ is a generator of this group and in particular is nonzero.
Naturality implies that
\begin{equation*}
% \label{Eq_0821_UseKT}
{\displaystyle{
 \underset{g \in F}{\smile}
   \Act_{g}^*(\eta) |_{\bigcap_{h \in F} \Act_{h}^{-1} (Y)}
 = p^* (\mu). }}
\end{equation*}
Now $t^* (p^* (\mu)) = \mu \neq 0$,
so $p^* (\mu) \neq 0$.
This is the claim,
and part~(\ref{Prop_mcid_for_full_shifts_pI}) is proved.

Part~(\ref{Prop_mcid_for_full_shifts_mcidEst})
is immediate from part~(\ref{Prop_mcid_for_full_shifts_pI}).

We prove~(\ref{Prop_mcid_for_shifts_RFin}).
The case $\dim (Z) = 0$ is easy,
so suppose $\dim (Z) > 0$.
Since $G$ has a quotient which is infinite but residually finite,
$G$ has arbitrarily large finite quotients.
For a finite set $S \subset G$,
we denote by $q_S \colon Z^G \to Z^S$
the projection onto the coordinates given by $S$.
Note that $Z^S$ is a CW complex of dimension $\crd(S) \cdot \dim(Z)$.

Fix $\delta>0$.
We will prove that there exists an even integer $k$
such that $\mcid_k (\Act; R) > \dim (Z) - \delta$.

Pick a normal group $N \lhd G$ of finite index
such that $[G:N] > 2/\delta$.
Let $S \subset G$ be a set of coset representatives for $G/N$,
with $1 \in S$.
Let $k$ be the largest even integer less than
$\crd(S) \cdot \dim(Z) = \dim(Z^S)$.
As before,
but with $Z^S$ in place of~$S$,
fix an embedding of $S^k$ into $Z^S$ and a finite open cover $\cV$
of $S^k$ in~$Z^S$
such that $\ch^k (S^k; \cV, R) = \ch^k (S^k; R)$,
set $Y = q_S^{-1} (S^k)$,
fix an isomorphism $\ch^k (S^k; R) \to R$,
and let $\et_0 \in \ch^k (S^k; R)$ correspond to $1 \in R$.
Then $\et_0 \in \ch^k (S^k; \cV; R)$.
Set $\cU = q_S^{-1} (\cV)$,
and set $\eta = (q_S |_Y)^* (\eta_0)$.
As before, we have $\calD_Y (\cU) = k$.

Let $G_0 \subset G$ be a finite set and let $\eps > 0$.
Define
\[
M = N \cap \bigl\{ s^{-1} g t \mid
          {\mbox{$g \in G_0$ and $s, t \in S$}} \bigr\},
\]
which is a finite subset of $N$.
Since $N$ is amenable,
there is a finite nonempty $(M, \ep)$-invariant subset $F_0 \S N$.
Define $F = S F_0$.
Then $\crd (F) = \crd (S) \crd (F_0)$.
We claim that $F$ is $(G_0, \eps)$-invariant.
So let $g \in G_0$.
For each $t \in S$ there is a unique $s (g, t) \in S$
such that $g t \in s (g, t) N$.
Then $g t F_0 \S s (g, t) N$,
so $g t F_0 \cap F \S s (g, t) N$,
whence $g t F_0 \cap F = g t F_0 \cap s (g, t) F_0$.
Now $s (g, t)^{-1} g t \in M$, so
\begin{align*}
\crd ( g t F_0 \cap F )
& = \crd ( g t F_0 \cap s (g, t) F_0 )
\\
& = \crd ( s (g, t)^{-1} g t F_0 \cap F_0 )
  > (1 - \ep) \crd (F_0).
\end{align*}
One can check that if $t_1, t_2 \in S$ and $s (g, t_1) = s (g, t_2)$,
then $t_2^{- 1} t_1 \in N$, whence $t_1 = t_2$.
Therefore the sets $g t F_0 \cap F$, for $t \in S$, are disjoint.
Summing over all $t \in S$ now gives
\begin{align*}
\crd ( g t F \cap F )
& = \sum_{t \in S} \crd ( g t F_0 \cap s (g, t) F_0 )
\\
& > (1 - \ep) \crd (S) \crd (F_0)
  = (1 - \ep) \crd (F),
\end{align*}
as claimed.

Since the translates of $S$ under $F_0$ are disjoint,
reasoning similar to that used in the proof of
part~(\ref{Prop_mcid_for_full_shifts_pI})
shows that
\[
\displaystyle
 \underset{g \in F_0}{\smile}
    \Act_{g}^* (\eta) |_{\bigcap_{h \in F_0} \Act_{h}^{-1} (Y)}
 \neq 0.
\]
Also
\begin{equation*}
\begin{split}
\frac{k \cdot \crd (F_0)}{\crd (F)}
& \geq \frac{\bigl[ \crd (S) \dim (Z) - 2 \bigr]
                          \cdot \crd (F_0)}{\crd (F)}
\\
& = \dim (Z) - \frac{2}{\crd (S)}
  > \dim (Z) - \delta.
\end{split}
\end{equation*}
It follows that $\mcid_k (\Act; R) \geq \dim (Z) - \dt$,
as desired.

Since $\dt > 0$ is arbitrary,
we get $\mcid (\Act; R) \geq \dim (Z)$.
On the other hand,
it is easy to deduce from Corollary~4.2 of~\cite{CrtKgr}
% (since the shift on $([0, 1]^{\dim (Z)})^G$ is a subsystem)
that $\mdim (T) \leq \dim (Z)$.
It now follows from Proposition~\ref{Prop_mcid_bounded_by_mdim}
that $\mcid (\Act; R) = \dim (Z)$, as required.
\end{proof}

One of the main sources of examples of minimal homeomorphisms
with nonzero mean dimension is subshifts.
Krieger established some lower bounds in \cite{Krieger2009},
generalizing the case of $\Z$
which was discussed in \cite{LinWeiss2000},
and Dou in \cite{Dou2017} constructed more specific examples
in which one can compute mean dimension precisely.
Here we obtain a related lower bound
for mean cohomological independence dimension for subshifts.

For convenience, we give two definitions.
The first is standard.
The second is intended only for use in this paper,
and identifies a feature which is a useful hypothesis
and which is common among constructions in the literature
of minimal subshifts.

\begin{Def}\label{D_0908_Density}
Let $G$ be a countable amenable group.
A {\emph{F{\o}lner sequence}} in $G$ is a sequence
$(F_n)_{n \in \N}$ of nonempty finite subsets $F_n \S G$
such that for all $g \in G$, we have
\[
\lim_{n \to \I} \frac{\crd (g F_n \triangle F_n)}{\crd (F_n)} = 0.
\]
For a subset $J \S G$, we define its {\emph{density}}
$\dt (J)$ to be the supremum over all F{\o}lner sequences
$(F_n)_{n \in \N}$ in $G$ of
\[
\limsup_{n \to \I} \frac{\crd (J \cap F_n)}{\crd (F_n)}.
\]
\end{Def}

\begin{Def}\label{D_0908_Full}
Let $G$ be a discrete group, let $Z$ be a set,
and let $\Act$ be the shift action of $G$ on $Z^G$.
Let $X \S Z^G$ be $\Act$-invariant.
We say that a subset $J \S G$ is {\emph{$X$-unconstrained}}
(or just {\emph{unconstrained}} if $X$ is understood)
if there is a point $z = (z_g)_{g \in G} \in X$
such that for any $x = (x_g)_{g \in G} \in Z^G$,
if $x_g = z_g$ for any $g \not \in J$ then $x \in X$.
We call such a point $z$ a {\emph{witness}} for the
$X$-unconstrainedness of~$J$.
\end{Def}

If $X = Z^G$ then $G$ itself is $X$-unconstrained.
Subsets of unconstrained sets in~$G$ are unconstrained.
We are interested in shift invariant subsets $X \S Z^G$.
In this case, if $J$ is $X$-unconstrained, with witness~$z$,
and $g \in G$, then $g J$ is also $X$-unconstrained,
with witness~$T_g (z)$.

\begin{Prop}\label{P_0908_ExistFull}
Let $Z$ be a polyhedron, let $\rh \in (0, 1)$,
and let $X \S Z^G$ be the minimal subshift of the shift $\Act$ on $Z^G$
constructed in Section~4 of~\cite{Dou2017}
to satisfy $\mdim (\Act) = \rh \dim (P)$.
Then there is an $X$-unconstrained subset $J \S G$
such that $\dt (J) \geq \rh$.
\end{Prop}

\begin{proof}
We use the set $J \subset G$ constructed in
Section~4.2 of~\cite{Dou2017}.
It is proved there that $\dt (J) \geq \rh$.
The second half of Section~4.2 of~\cite{Dou2017}
proves the existence of $z = (z_g)_{g \in G} \in X$
such that for any $x = (x_g)_{g \in G} \in Z^G$,
if $x_g = z_g$ for any $g \not \in J$ then $x \in X$,
that is, $z$ is a witness for the $X$-unconstrainedness of~$J$.
\end{proof}

The fact that the set $J$ in~\cite{Dou2017} is $X$-unconstrained
is crucial in the proof there of the lower bound for $\mdim (\Act)$.
The assumption in Proposition \ref{P_0908_ExistFull}
that $Z$ is a polyhedron follows \cite{Dou2017}.
However, we assume
that the same holds if we assume that $Z$ is any finite CW-complex.

\begin{Prop}\label{Prop_mcid_for_subshifts}
Let $Z$ be a finite CW-complex.
Let $X$ be a closed $G$-invariant subset of $Z^G$,
with the shift action $\Act$.
Let $J \subset G$ be an $X$-unconstrained subset
with witness $z = (z_g)_{g \in G} \in X$.
Let $\rho = \delta(J)$ be the density of $J$.
Let $R$ be a principal ideal domain.
Then for any even integer $k < \dim(Z)$
we have $\mcid_k (\Act; R) \geq k \rho$.
Moreover,
\[
\mcid (\Act |_X; R)
 \geq 2 \rho \left\lfloor \frac{\dim(Z) - 1}{2} \right\rfloor.
\]
\end{Prop}

\begin{proof}
It suffices to prove the first statement.
The proof is similar to that of Proposition
\ref{Prop_mcid_for_full_shifts}(\ref{Prop_mcid_for_full_shifts_pI}).
We may assume $\dim (Z) > 0$.
We may assume without loss of generality that $1 \in J$.

Let $k$ be a nonnegative even integer such that $k < \dim(Z)$.
As $Z$ has a cell of dimension greater than $k$,
we can embed the sphere $S^k$ into $Z$.
As in the proof of Proposition
\ref{Prop_mcid_for_full_shifts}(\ref{Prop_mcid_for_full_shifts_pI}),
fix an embedding of $S^k$ into $Z$ and a finite open cover $\cV$
of $S^k$ in~$Z$
such that $\ch^k (S^k; \cV, R) = \ch^k (S^k; R)$,
fix an isomorphism $\ch^k (S^k; R) \to R$,
and let $\et_0 \in \ch^k (S^k; R)$ correspond to $1 \in R$.
Then $\et_0 \in \ch^k (S^k; \cV; R)$.
Let $q \colon Z^G \to Z$ be the projection onto the coordinate $g=1$.
Then $q(X) = Z$.
Set $Y = q^{-1} (S^k) \cap X$,
set $\cU = q^{-1} (\cV) \cap X$, and set $\eta = (q |_Y)^* (\eta_0)$.
% For any $S \subset G$, we denote by $q_S \colon Z^G \to Z^S$
% the projection onto the coordinates given by $S$.
% By assumption, $q_J(X) = Z^J$.
As before, and relying on the next claim, we have $\calD_Y (\cU) = k$.

We claim that for any finite set $F \subset G$, the element
\[
\nu = {\displaystyle{
 \underset{g \in F \cap J}{\smile}
   \Act_{g}^*(\eta) |_{\bigcap_{h \in F \cap J} \Act_{h}^{-1} (Y)} }}
\]
is nonzero.
Letting $F$ run through suitable F{\o}lner sequences,
the definition will give $\mcid_k (\Act |_X; r) \geq k \rho$,
as required.

To prove the claim,
define maps
\[
t \colon
 (S^k)^{F \cap J} \to \bigcap_{g \in F \cap J} \Act_{g^{-1}} (Y)
\andeqn
p \colon
 \bigcap_{g \in F \cap J} \Act_{g^{-1}} (Y) \to (S^k)^{F \cap J}
\]
as follows.
Take $p$ to be the restriction of the projection map
\[
(S^k)^{F \cap J} \times Z^{G \smallsetminus (F \cap J)}
   \to (S^k)^{F \cap J}.
\]
For $x \in (S^k)^{F \cap J}$
define
\[
t (x)_g = \begin{cases}
   x_g & \hspace*{1em} g \in F \cap J
        \\
   z_g & \hspace*{1em} g \in G \smallsetminus (F \cap J).
\end{cases}
\]
The conditions on $z$ imply that $t (x)$ as defined here
really is in~$X$,
and it now follows that
$t (x) \in \bigcap_{g \in F \cap J} \Act_{g^{-1}} (Y)$.
Then $p \circ t = \id_{(S^k)^F}$.
Following the proof of Proposition
\ref{Prop_mcid_for_full_shifts}(\ref{Prop_mcid_for_full_shifts_pI}),
set $\mu = \prod_{g \in F \cap J} \et_0$,
use naturality to get $\nu = p^* (\mu)$,
and use $\mu \neq 0$ and $t^* (p^* (\mu)) = \mu$
to get $p^* (\mu) \neq 0$.
This is the claim.
\end{proof}

\begin{Rmk}\label{R_0826_Obstruct}
One can use Proposition \ref{Prop_mcid_for_subshifts},
Remark \ref{Rmk:mcid for invariant subsets}, and
Proposition
 \ref{Prop_mcid_for_full_shifts}(\ref{Prop_mcid_for_shifts_RFin})
to obtain obstructions for embedding various subshifts in full shifts.
This does not entirely bypass the use of mean dimension,
as we used mean dimension
in the proof of
Proposition
 \ref{Prop_mcid_for_full_shifts}(\ref{Prop_mcid_for_shifts_RFin}).
\end{Rmk}

\section{Lower bounds for the radius of comparison}\label{Sec_0908_LB}

The goal of this section is to show that
for any countable amenable group $G$,
the radius of comparison of $C(X) \rtimes_{\Act} G$
is bounded below by $\mcid_k(\Act;\Q)-1-k/2$.

We first require a lemma in algebraic topology,
motivated by ideas which appear
in the proof of Lemma 2.13 of \cite{ElliottNiu2013}.
We denote by $K(n,\Q)$ the Eilenberg-MacLane spaces, that is,
the classifying spaces for \v{C}ech cohomology
with coefficients in $\Q$.
This means that for a compact metrizable space $Y$,
the \v{C}ech cohomology group $\ch^n(Y;\Q)$ is naturally isomorphic to
the set $[Y, \, K (n, \Q)]$ of homotopy classes of maps
from $Y$ to the Eilenberg-MacLane space.
The space $K(n,\Q)$ is unique up to homotopy equivalence.

\begin{Rmk}\label{R_0826_Chern}
The Chern classes of a vector bundle $E$
over a compact metrizable space $X$
are elements $c_q(E)$ of $\ch^{2q}(X;\Z)$.
(See Theorem V.3.15 and Remark V.3.21 of~\cite{KaroubiBook}.)
The total Chern class is $c (E) = 1 + c_1 (E) + c_2 (E) + \cdots.$
When we work in rational cohomology,
we use the rational version
\[
c_q^{\Q} (E)
 = c_q (E) \otimes 1_{\Q}
 \in \ch^{2 q} (X; \Z) \otimes \Q
 \cong \ch^{2 q} (X; \Q)
\]
and, similarly, $c^{\Q} (E) \in \ch^{*} (X; \Q)$.
\end{Rmk}

\begin{Lemma}\label{Lemma:zero lower classes}
Let $Z$ be a finite CW complex.
Suppose $q$ is a positive integer
and $\eta \in \ch^{2q}(Z;\Q)$ is a nonzero cohomology class.
Then there exists a positive integer $m$
and a vector bundle $E$ over $Z$
such that the $q$-th rational Chern class of $E$ is
$c_q^{\Q} (E) = m \eta$ and such that $c_j^{\Q} (E) = 0$
for all $j \in \{ 1, 2, \ldots, q - 1 \}$.
\end{Lemma}

\begin{proof}
By the construction
of Eilenberg-MacLane spaces in Section~4.2 of \cite{HatcherBook},
we can choose a model for $K (2 q; \Q)$
with no cell of dimension strictly between $0$ and $2 q$.
Represent the cohomology class $\eta$
as a function $f \colon Z \to K(2q;\Q)$.
Since $Z$ is compact,
the image of $Z$ is contained in a finite subcomplex $Z'$
of $K(2q;\Q)$ (by \cite[Proposition A.1]{HatcherBook}).
% Since $K (2 q; \Q)$ has no cells
% of dimension strictly between $0$ and $2 q$, neither does~$Z'$.
Then $Z'$ also has no cells of dimension strictly between $0$ and $2 q$.
Let $\iota \colon Z' \to K(2q;\Q)$ be the inclusion map,
and let $f' \colon Z \to Z'$ be the map $f$,
restricting the codomain to be $Z'$, so that $f = \iota \circ f'$.
Denote by $[\iota] \in \ch^{2q}(Z';\Q)$
the class represented by $\iota$.
Then $\eta = (f')^*([\iota])$.

We refer the reader to
the proof of Proposition IV.7.11 of~\cite{KaroubiBook}
% page 255
for the definition of the Newton polynomials $Q_n$ for
$n=1,2,3,\ldots$,
and to Section V.3 of~\cite{KaroubiBook}
for a discussion of the Chern character.
If $X$ is a compact metrizable space
and $E$ is a vector bundle over $X$,
recall (V.3.19 and V.3.22 in~\cite{KaroubiBook})
that we define the $n$-th component of the Chern character
of a vector bundle $E$ by
\[
\Ch_n (E)
 = \frac{1}{n!} Q_n \bigl( c_1^{\Q} (E), \, c_2^{\Q} (E), \,
          \ldots, \,c_n^{\Q} (E) \bigr)
 \in \ch^{2 n} (X; \Q),
\]
and that $\Ch (E) = \sum_{n = 0}^{\infty} \Ch_n (E)$
(with $\Ch_0 (E)$ taken to be the dimension of $E$).
The Chern character gives rise to an isomorphism
$\Ch \colon K^0 (X) \otimes \Q \to \ch^{\mathrm{even}} (X; \Q)$
(Theorem V.3.25 of~\cite{KaroubiBook}).

Returning to the situation in the first paragraph,
we know that there exist vector bundles
$F_1,F_2,\ldots,F_s$ over $Z'$
and rational numbers $r_1,r_2,\ldots,r_s$ such that
$\Ch_q \bigl( \sum_{l = 1}^s r_l [F_l] \bigr) = [\iota]$.
However, since $Z$ has no cells
of dimension strictly between $0$ and $2q$,
we know that $c_j^{\Q} ([F_l]) = 0$ for $j=1,2,\ldots,q-1$
and $l=1,2,\ldots,s$.
Pick a strictly positive integer $p$ such that $pr_l \in \Z$
for $l=1,2,\ldots,s$.
We thus have
$\Ch_q ( \sum_{l = 1}^s pr_l [F_l]) = p [\iota]$.
Some of those coefficients may be negative.
To overcome that,
for each $l$ such that $r_l < 0$,
replace $F_l$ with its complement
in a sufficiently large trivial bundle,
and replace $r_l$ with $- r_l$.
This change replaces $\sum_{l = 1}^s pr_l [F_l]$
with $[H] + \sum_{l = 1}^s pr_l [F_l]$
for some trivial bundle~$H$.
It has the effect of changing the value of $\Ch_0$,
which we do not care about, but not the value of $\Ch_q$.
We may thus assume without loss of generality
that $r_l>0$ for $l=1,2,\ldots,s$.

The formal sum $\sum_{l=1}^s pr_l[F_l] $
can be replaced now by the direct sum,
so we have a vector bundle $F$ over $Z'$ which satisfies
$\Ch_q (F) = p [\iota]$ and $c_j^{\Q} (F) = 0$
for $j=1,2,\ldots,q-1$.
Thus, as all but one of the terms in the Newton expression vanish,
we have (justification for the second step below)
\begin{align*}
\Ch_q (F)
& = \frac{1}{q!} Q_q \bigl( c_1^{\Q} (F), \, c_2^{\Q} (F), \,
              \ldots, \, c_q^{\Q} (F) \bigr)
\\
& = \frac{1}{q!} \cdot (-1)^{q - 1} q c_q^{\Q} (F)
  = \frac{(-1)^{q - 1}}{(q - 1)!} c_q^{\Q} (F).
\end{align*}
The second step can be deduced from (2.11$'$)
on page 23 of~\cite{Macdonald1995} by rearranging terms;
see pages 19 and~23 there for the notation.
Let $F'$ be the direct sum of $(q-1)!$ copies of $F$ if $q$ is odd,
and the complement in a large trivial bundle
of the direct sum of $(q-1)!$ copies of $F$ if $q$ is even.
Since the total Chern class is multiplicative
and $c (F) = 1 + c_q (F) + c_{q + 1} (F) + \cdots$,
we have $c_q^{\Q} (F') = (q - 1)! p [\iota]$.
Now, set $m = (q-1)!p$, and let $E = (f')^*(F')$.
Then, by naturality,
we have $c_q^{\Q} (E) = m \eta$
and $c_j^{\Q} (E) = 0$ for $j=1,2,\ldots,q-1$.
\end{proof}

We can now prove our main theorem.

\begin{Thm}\label{T_0826_Main}
Let $G$ be a countable amenable group,
let $X$ be a compact metrizable space,
and let $\Act$ be an action of $G$ on $X$.
Let $k$ be an even integer,
and let $m$ be the greatest integer
with $m < \mcid_k (\Act; \Q)$.
Then $\rc (C (X) \rtimes_{\Act} G) \geq m - k/2$.
If $C (X) \rtimes_{\Act} G$ is simple,
then $\rc (C (X) \rtimes_{\Act} G) > m - k/2$.
\end{Thm}

In particular,
$\rc (C (X) \rtimes_{\Act} G) \geq \mcid_k (\Act; \Q) - 1 - k/2$.

\begin{proof}[Proof of Theorem~\ref{T_0826_Main}]
Fix an even integer~$k$.
Fix $\eps_0>0$.

Find a closed subset $Y \subset X$,
a finite open cover $\cU$ of $Y$ in~$X$ such that
% $k = 2 \left\lfloor \frac{\calD (\cU)}{2} \right\rfloor$,
% $\calD (\cU) = k$ or $\calD (\cU) = k + 1$,
$\calD_Y (\cU) \in \{ k, \, k + 1 \}$,
and a cohomology class $\eta \in \ch^k(Y;\cU;\Q)$ such that
for any finite subset $G_0 \subset G$ and any $\delta>0$
there exists a finite $(G_0,\delta)$-invariant subset $F$
and a subset $F_0 \subset F$ satisfying:
\begin{enumerate}
\item\label{T_0826_Main_Nonzero}
$\displaystyle
 \underset{g \in F_0}{\smile}
   \Act_g^* (\eta) |_{\bigcap_{h \in F_0} \Act_{h}^{-1} (Y)}
 \neq 0$.
\item\label{T_0826_Main_Folner}
\setcounter{TmpEnumi}{\value{enumi}}
$\displaystyle
 \frac{k \cdot \crd (F_0)}{\crd (F)} > \mcid_k(\Act;\Q) - \eps_0$.
\end{enumerate}

Pick a finite open cover $\cV$ of $Y$ in~$X$
such that $\cV \prec_Y \cU$ and $\ord (\cV) = \calD_Y (\cU)$.
Recall that $\cN (\cV)$ is the nerve of~$\cV$.
By the construction of \v{C}ech cohomology,
$\eta$ is also in $\ch^k(Y;\cV;\Q)$,
and can be obtained as a pullback of a cohomology class
$\eta' \in \ch^k(\cN(\cV);\Q)$
by a map $f \colon Y \to \cN(\cV)$.
Use Lemma \ref{Lemma:zero lower classes}
to choose a vector bundle $E'$ over $\cN(\cV)$
such that $c_{k / 2}^{\Q} (E') = M \eta'$ for some nonzero integer~$M$
and $c_j^{\Q} (E') = 0$
for all $j \in \{ 1, 2, \ldots, k / 2 - 1 \}$.
Since $\dim ( \cN ( \cV)) < k + 2$,
we also have $c_j^{\Q} (E') = 0$
for all $j \in \{ k / 2 + 1, \, k / 2 + 2, \ldots \}$.
Therefore $c^{\Q} (E') = 1 + M \et'$.
As we can replace $\eta$ by any nonzero scalar multiple
of it and retain the same properties,
we may use $M\eta$ in place of $\eta$,
so we may lighten notation and assume that $c(E') = 1+\eta'$.
Since $\dim(\cN(\cV)) \leq k+1$,
by subtracting trivial bundles
(see Theorem 9.1.2 of~\cite{Husemoller}),
we can also assume that $\rank(E') = k/2$.
Now set $E = f^*(E')$.
We have $c(E) = 1+\eta$ by naturality, and $\rank(E) = k/2$.

Let $L$ be the dimension of some trivial bundle
which has $E$ as a direct summand, and
let $q \in M_L (C (Y))$ be the projection onto $E$.
Let $a \in M_L(C(X))$ be a positive contraction
such that $a |_Y = q$
and $\rank(a(x)) \leq \rank(E)$ for all $x \in X$.
Let $b \in M_{\infty} (C (X))$ be a constant projection.
Suppose that $a \precsim_{C (X) \rtimes_{\Act} G} b$.
We are going to prove that
this implies $\rank (b) \geq \mcid_k (\Act; \Q) - 2 \ep_0$.

Increasing $L$ if needed,
we may assume that $b \in M_{L} (C (X))$.
There is $c_0 \in M_{\infty} (C (X) \rtimes_{\Act} G)$
such that $\| c_0^* b c_0 - a\| <  1/8$.
Replacing $c_0$ by its cutdown by
$1_{M_{L} (C (X) \rtimes_{\Act} G)}$,
we may assume that
\[
c_0
 \in M_{L} (C (X) \rtimes_{\Act} G)
 \cong (M_{L} \otimes C (X)) \rtimes_{\id_{M_{L}} \otimes \Act} G.
\]
Then there is $c$ in the algebraic crossed product
such that
\begin{equation}\label{Eq_0829_cbc}
\| c^* b c - a \| <  \frac{1}{4}.
\end{equation}
%
% Being in the algebraic crossed product means that,
With $u_g \in C (X) \rtimes_{\Act} G$ being the standard
unitary corresponding to $g \in G$,
there are a finite subset $G_0 \subset G$
and elements $c_g \in M_L (C (X))$ for $g \in G_0$
such that $c = \sum_{g \in G_0} c_g (1_{M_L} \otimes u_g)$.
Increasing the size of the set~$G_0$,
we may assume that
\begin{equation}\label{Eq_0829_Inverse}
G_0 = \bigl\{ g^{-1} \mid g \in G_0 \bigr\}.
\end{equation}
Choose $\delta >0$ such that
\[
\delta \cdot (\| c \| + 1) \cdot \sum_{g \in G_0} \|c_g\|
 < \frac{1}{16}.
\]

Choose a nonempty finite $(G_0,\delta)$-invariant subset $G_1$ of $G$
with $1 \in G_1$.
Fix $\eps_1>0$ which satisfies
\begin{equation}\label{Eq_0830_CardEst}
\frac{ \mcid_k (\Act; \Q) - \eps_0 - k \eps_1}{1 + \eps_1}
 > \mcid_k (\Act; \Q) - 2 \eps_0.
\end{equation}
Set $\ep_2 = \ep_1 / \crd (G_1)$.
Find a finite nonempty
$\bigl( G_1 \cup (G_1)^{-1}, \, \eps_2 \bigr)$-invariant subset
$F \subset G$
such that there exists a subset $F_0 \subset F$ for which:
\begin{enumerate}
\setcounter{enumi}{\value{TmpEnumi}}
\item\label{Eq_0829_Prod}
$\displaystyle
 \underset{g \in F_0}{\smile}
 \Act_g^* (\eta) |_{\bigcap_{h \in F_0} \Act_{h}^{-1} (Y)} \neq 0$.
\item\label{Eq_0829_Folner}
$\displaystyle \frac{k \cdot  \crd (F_0)}{\crd (F)}
 > \mcid_k(\Act;\Q) - \eps_0$.
\end{enumerate}
For any $h_0 \in G$,
if we replace $F$ and $F_0$ by $F h_0^{- 1}$ and $F_0 h_0^{- 1}$,
then we replace the class in~(\ref{Eq_0829_Prod})
by its image under $T_{h_0}^*$,
which is still nonzero.
Doing this with some $h_0 \in F$,
we may therefore assume that $1 \in F$.

Set
\[
K = \bigcap_{h \in G_1} h F.
\]
Then
\begin{equation}\label{Eq_0830_FppEst}
\begin{split}
\crd (F \smallsetminus K)
& = \crd (F) - \crd (K)
  \leq \sum_{h \in G_1} \crd ( F \smallsetminus [F \cap h F])
\\
& < \crd (G_1) \ep_2 \crd (F)
  = \ep_1 \crd (F).
\end{split}
\end{equation}

Let $\Delta \colon G \to [0,1]$ be the function
\[
\begin{split}
\Delta (g)
& = \frac{1}{\crd (G_1)} (\chi_{G_1} * \chi_F) (g)
\\
& = \frac{1}{\crd (G_1)} \sum_{s \in G} \chi_{G_1} (s) \chi_F (s^{-1} g)
  = \frac{1}{\crd (G_1)} \sum_{h \in G} \chi_{G_1} (g h^{-1}) \chi_F (h)
\end{split}
\]
for $g \in G$.
Note that
\begin{equation}\label{Eq_0901_Dt}
\Delta(g)
% = \frac{1}{\crd (G_1)} \sum_{h \in G} \chi_{G_1} (h) \chi_F (h^{-1} g)
% = \frac{1}{\crd (G_1)} \sum_{h \in G_1} \chi_{F} (h^{-1} g).
  = \frac{1}{\crd (G_1)} \sum_{h \in F} \chi_{G_1} (g h^{-1})
  = \frac{1}{\crd (G_1)} \sum_{s \in G_1} \chi_{F} (s^{-1} g).
\end{equation}
Then for $t \in G_0$ and $g \in G$ we have,
using~(\ref{Eq_0829_Inverse})
and $(G_0, \dt)$-invariance of~$G_1$ at the last step,
\begin{equation}
\label{eqn_Delta_almost_invariant}
\begin{split}
|\Delta (t^{-1} g) - \Delta(g)|
& = \left | \frac{1}{\crd (G_1)} \sum_{h \in F}
     \bigl( \chi_{g^{-1} t G_1} (h^{-1})
              - \chi_{g^{-1} G_1} (h^{-1}) \bigr) \right |
\\
& \leq  \frac{\crd \bigl( (t G_1 \smallsetminus G_1)
      \cup (G_1 \smallsetminus t G_1) \bigr)}{\crd (G_1)}
  < 2 \delta.
\end{split}
\end{equation}
We further claim that $\Delta(g) = 1$ for all $g \in K$.
To see this, notice that,
by~(\ref{Eq_0901_Dt}),
we have $\Delta(g) = 1$
if and only if $s^{-1} g \in F$ for all $s \in G_1$,
that is, $g \in K$.

Let $\af$ be the corresponding action of $G$ on $C (X)$,
that is, $\alpha_g (f) (x) = f (\Act_g^{-1} (x))$
for $g \in G$, $f \in C (X)$, and $x \in X$.
Let $l^2 (G) \otimes C (X)$ be the usual Hilbert $C (X)$-module,
and write its elements as functions $\xi \colon G \to C (X)$
such that $\sum_{g \in G} \xi (g)^* \xi (g)$ converges in $C (X)$.
For any Hilbert module~$H$,
let $B (H)$ denote the $C^*$-algebra of adjointable operators on~$H$.
We view $C(X) \rtimes_{\Act}G $
as embedded in $B(l^2(G) \otimes C(X))$
in the standard way, that is,
if $f \in C(X)$, $\xi \in l^2(G) \otimes C(X)$, and $g, h \in G$, then
\[
(f \cdot \xi) (g) = \alpha_{g^{-1}} (f) \cdot \xi (g)
\andeqn
(u_h \cdot \xi) (g) = \xi (h^{-1} g).
\]
We define a multiplication operator $d_0 \in B(l^2(G) \otimes C(X))$ by
\[
(d_0 \xi) (g) = \Delta (g) \xi (g)
\]
for $g \in G$.
By~(\ref{eqn_Delta_almost_invariant}),
for any $t \in G_0$,
$\| u_t d_0 u_t^* - d_0 \| < 2 \delta$,
whence $\| u_t d_0 - d_0 u_t \| < 2 \delta$.
Set $\td = d_0 \otimes 1_{M_L}$.
% Define $\delta_1 = 2 \delta \sum_{g \in G_0} \|c_g\|$.
It follows that
\begin{equation}\label{Eq_0829_CommEst}
\| c \td - \td c \|
 < 2 \delta \sum_{g \in G_0} \|c_g\|
 < \frac{1}{8 (\| c \| + 1)}.
\end{equation}
%
% and recall that $\delta$ was chosen so that $\delta_1< 1/4$.

Notice that $\supp(\Delta) = G_1F$.
Thus, we can view $\td  M_{L}(C(X) \rtimes_{\Act} G) \td$
as included in $B \bigl( l^2(G_1F) \otimes M_L \otimes C(X) \bigr)$,
with $l^2 (G_1 F) \otimes M_L \otimes C (X)$
regarded as a Hilbert $M_L \otimes C (X)$-module.
Since $G_1 F$ is a finite set,
this is a matrix algebra over $C (X)$.

Since, in particular,
$F$ is $\bigl( (G_1)^{-1}, \, \eps_2 \bigr)$-invariant,
we have
\begin{equation}\label{Eq_0830_FpF}
\crd (G_1 F)
 < \bigl( 1 + \eps_2 \crd (G_1) \bigr) \crd (F)
 = ( 1 + \eps_1 ) \crd (F).
\end{equation}
Set
\[
c' = \td^{1/2} c \td^{1/2}
   \in B \bigl( l^2 (G_1 F) \otimes M_L \otimes C (X) \bigr).
\]
Then, at the third step
using (\ref{Eq_0829_CommEst}) on the first term
and (\ref{Eq_0829_cbc}) on the second term,
\begin{equation*}
\begin{split}
& \bigl\| (c')^* \td^{1/2} b \td^{1/2} c'
               - \td^{3/2} a \td^{3/2} \bigr\|
\\
& \hspace*{3em} {\mbox{}}
 \leq
 \bigl\| \td^{1/2} c^* \td b \td c \td^{1/2}
                - \td^{3/2} c^* b \td c \td^{1/2} \bigr\|
   + \bigl\| \td^{3/2} c^* b \td c \td^{1/2}
                - \td^{3/2} c^* b c \td^{3/2} \bigr\|
   \\
& \hspace*{6em} {\mbox{}}
     + \bigl\| \td^{3/2} c^* b c \td^{3/2}
                - \td^{3/2} a \td^{3/2} \bigr\|
\\
& \hspace*{3em} {\mbox{}}
 \leq 2 \| c \| \| d c - c d \| + \| c^* b c - a \|
  < 2 \left( \frac{1}{8} \right) + \frac{1}{4}
  = \frac{1}{2}.
\end{split}
\end{equation*}
Now, under our identification of $a$
as an element in the crossed product,
$\td^{3/2}a\td^{3/2}$ is a diagonal operator
on the Hilbert $M_L \otimes C (X)$-module
$l^2 (G_1 F) \otimes M_L \otimes C (X)$,
and for any $g \in K$, since $\Delta(g) = 1$,
the $g$-th diagonal entry
is simply $(\alpha_{g^{-1}} \otimes \id_{M_L}) (a)$.
Likewise, as $b$ is invariant under the group action,
$\td^{3/2}b \td^{3/2}$ is a diagonal matrix
whose diagonal entries
are scalar multiples of the constant projection $b$.
We restrict all these diagonal entries
(which are matrix valued functions on $X$)
to $Y' = \bigcap_{g \in F_0 \cap K} \Act_{g}^{-1} (Y)$.

Let $p \in B \bigl( l^2 (G_1 F) \otimes M_L \otimes C (X) \bigr)$
be the diagonal projection
whose diagonal $g$-th entry is $1$ if $g \in K$ and zero otherwise.
We obtain
\[
\bigl\| p \td^{3/2} a \td^{3/2} p
        - p (c')^* \td^{1/2} b \td^{1/2} c' p \bigr\|
 < \frac{1}{2}.
\]
This remains true after restricting to $Y'$.
Note that $p\td^{3/2}a\td^{3/2}p = pap$.
Thus, the projection $pap|_{Y'}$
is Murray-von-Neumann subequivalent
to the cutdown of $b$ to
$B \bigl( l^2 (G_1F) \otimes M_L \otimes C (X) \bigr)$,
restricted to $Y'$,
which is a constant projection of rank $\crd(G_1F)\cdot \rank(b)$.
Notice that
$pap|_{Y'}$ is the projection onto
\[
{\widetilde{E}}
 = \bigoplus_{g \in F_0 \cap K}
    \Act_g^*(E) |_{\bigcap_{h \in F_0 \cap K} \Act_{h}^{-1} (Y)}.
\]
Now,
\[
c \bigl( {\widetilde{E}} \bigr)
 = {\displaystyle{ \underset{g \in F_0 \cap K}{\smile}
       (1 + \Act_g^* (\eta))
           |_{\bigcap_{h \in F_0 \cap K} \Act_{h}^{-1} (Y)} }},
\]
so if ${\widetilde{E}} \oplus E'$ is a trivial bundle, then
\[
c (E')
 = c \bigl( {\widetilde{E}} \bigr)^{-1}
 = {\displaystyle{
    \underset{g \in F_0 \cap K}{\smile}
       (1 - \Act_g^* (\eta))
           |_{\bigcap_{h \in F_0 \cap K} \Act_{h}^{-1} (Y)} }}.
\]
In particular, $c_{k \cdot \crd(F_0 \cap K)/2}(E') \neq 0$,
so $\rank(E') \geq k \cdot \crd(F_0 \cap K)/2$.
Therefore ${\widetilde{E}}$ cannot embed
into a trivial bundle of rank less than $k \cdot \crd (F_0 \cap K)$.
Thus
\begin{equation}\label{Eq_0901_Rank_b}
\crd (G_1F) \cdot \rank (b)
 \geq k \cdot \crd (F_0 \cap K).
\end{equation}

Using (\ref{Eq_0829_Folner}) and~(\ref{Eq_0830_FppEst})
at the second step,
\[
\begin{split}
\frac{k \cdot \crd (F_0 \cap K)}{\crd (F)}
& \geq \frac{k \cdot \crd (F_0)}{\crd (F)}
          - \frac{k \cdot \crd (F \smallsetminus K)}{\crd (F)}
\\
& > \mcid_k (\Act; \Q) - \eps_0 - k \ep_1,
\end{split}
\]
so,
using (\ref{Eq_0901_Rank_b}) at the first step,
(\ref{Eq_0830_FpF}) at the second step,
and (\ref{Eq_0830_CardEst}) at the third step,
\[
\begin{split}
\rank (b)
& \geq \frac{k \cdot \crd (F_0 \cap K)}{\crd (F)}
       \left( \frac{\crd (F)}{\crd (G_1 F)} \right)
\\
& > \frac{\mcid_k (\Act; \Q) - \eps_0 - k \ep_1}{1 + \eps_1}
  > \mcid_k (\Act; \Q) - 2 \eps_0.
\end{split}
\]
This is what we set out to prove.

Now recall that $m $ is the greatest integer
with $m < \mcid_k (\Act; \Q)$.
Choose $\ep_0 > 0$ such that $\mcid_k (\Act; \Q) - m - 2 \ep_0 > 0$.
Make the choices above with this value of $\ep_0$,
but take $b$ to be a constant projection with $\rank (b) = m$.
Then $a \not\precsim_{C (X) \rtimes_{\Act} G} b$.
For any invariant tracial state $\tau$ on $C(X)$,
and hence for any tracial state $\tau$ on $C (X) \rtimes_{\Act} G$,
we have $d_{\tau}(b) = \rank(b)$
and $d_{\tau}(a) \leq \rank(a) = k/2$.
So $C (X) \rtimes_{\Act} G$ does not
have $(m - k/2)$-comparison.
Thus, $\rc (C (X) \rtimes_{\Act} G) \geq m - k/2$.

If $C (X) \rtimes_{\Act} G$ is simple,
then, since this algebra is also stably finite,
Proposition~6.3 of~\cite{Toms-rc}
implies that the set of real numbers $r$
such that $C (X) \rtimes_{\Act} G$
has $r$-comparison is closed.
So in fact $\rc (C (X) \rtimes_{\Act} G) > m - k/2$.
\end{proof}

Applying this result to the subshift of \cite{Dou2017}
and rounding several estimates,
we get the following result.
The ``loss factor'' $1 - \frac{1 - \rh}{\rh}$,
which is the main part of the difference between our estimate and the
conjectured value of $\rc (C (X) \rtimes_{\Act} G)$,
is $1$ if $\rh = 1$ and close to~$1$ if $\rh$ is close to~$1$,
but makes the estimate useless if $\rh \leq \frac{1}{2}$.
If we assume $\ch^k (Z; \Q) \neq 0$,
we can remove this factor.
See Corollary~\ref{C_0907_Shift_Sk}.

\begin{Cor}\label{C_0926_Shift_Gen}
Let $G$ be a countable amenable group,
let $Z$ be a polyhedron,
and let $\rh \in (0, 1)$.
Let $(X, \Act)$ be the minimal subshift of the shift on $Z^G$
constructed in Section~4 of~\cite{Dou2017}
to satisfy $\mdim (\Act) = \dim (Z) \rh$.
Then
\[
\rc (C (X) \rtimes_{\Act} G)
 > \frac{1}{2} \mdim (T) \left( 1 - \frac{1 - \rh}{\rh} \right) - 2.
\]
\end{Cor}

\begin{proof}
We may certainly assume $\dim (Z) > 0$.
Let $k$ be largest even integer with $k < \dim (Z)$,
so that $\dim (Z) - 2 \leq k \leq \dim (Z) - 1$.
Then $\mcid_k (T) \geq k \rh$
by Proposition~\ref{Prop_mcid_for_subshifts}.
Therefore $\rc (C (X) \rtimes_{\Act} G) > k \rh - 1 - k/2$
by Theorem~\ref{T_0826_Main}.
If $k = \dim (Z) - 2$,
using $\rh \leq 1$ at the last step, we get
\[
\begin{split}
\rc (C (X) \rtimes_{\Act} G)
& > [\dim (Z) - 2] \rh - 1 - \frac{\dim (Z) - 2}{2}
\\
& = \frac{1}{2} \dim (Z) \rh \left( 2 - \frac{1}{\rh} \right) - 2 \rh
  \geq \frac{1}{2} \dim (Z) \rh \left( 1 - \frac{1 - \rh}{\rh} \right)
         - 2.
\end{split}
\]
If $k = \dim (Z) - 1$, then instead
\[
\begin{split}
\rc (C (X) \rtimes_{\Act} G)
& > \frac{1}{2} \dim (Z) \rh \left( 1 - \frac{1 - \rh}{\rh} \right)
         - \rh - \frac{1}{2}
\\
& > \frac{1}{2} \dim (Z) \rh \left( 1 - \frac{1 - \rh}{\rh} \right)
         - 2.
\end{split}
\]
Now substitute $\dim (Z) \rh = \mdim (\Act)$.
\end{proof}

\begin{Exl}\label{E_0830_Large_rc}
For every infinite countable amenable group~$G$
and every $N > 0$,
there is a cube~$Z$
and a minimal subsystem $(X, T)$ of the shift action of $G$ on $Z^G$
such that $\rc (C (X) \rtimes_{\Act} G) > N$.

In Corollary~\ref{C_0926_Shift_Gen}
take $\rh = \frac{3}{4}$,
and choose an integer~$d$ with $d > 4 (N + 2)$.
The subshift $T$ in Section~4 of~\cite{Dou2017}
satisfies $\mdim (T) = d \rh$,
so
\[
\rc (C (X) \rtimes_{\Act} G)
 > \frac{1}{2} d \rh \left( 1 - \frac{1 - \rh}{\rh} \right) - 2
 = \frac{d}{4} - 2
 > N.
\]
\end{Exl}

\section{Symmetric mean cohomological independence
 dimension}\label{Sec_0908_Symm}

In this section, we define a variant of mean cohomological independence
dimension, which we use to obtain sharper lower bounds on the radius of
comparison for certain subshifts.

We recall the elementary symmetric polynomials:
for $n, r \in \{ 0, 1, 2, \ldots \}$
with $r \leq n$,
\[
\sm_r (x_1, x_2, \ldots, x_n)
 = \sum_{1 \leq j_1 < j_2 < \cdots < j_r \leq n}
   x_{j_1} x_{j_2} \cdots x_{j_r}
\]
is the elementary symmetric polynomial of degree $r$
in the $n$ variables $x_1, x_2, \ldots, x_n$.

\begin{Def}\label{N_0904_ElSymmP}
Let $k$ be an even integer, let $Y$ be a compact metrizable space,
let $R$ be a commutative unital ring,
let $F$ be a finite set,
and let $r \in \{ 0, 1, 2, \ldots, \crd (F) \}$.
If $(\et_g)_{g \in F}$ is a family of elements in $\ch^{k} (Y; R)$
indexed by~$F$, we define $\sm_r \bigl( (\et_g)_{g \in F} \bigr)$
as follows.
Set $n = \crd (F)$,
and enumerate $F$ as $\{ g_1, g_2, \ldots, g_n \}$.
Then define
\[
\sm_r \bigl( (\et_g)_{g \in F} \bigr)
 = \sum_{1 \leq j_1 < j_2 < \cdots < j_r \leq n}
   \et_{g_{j_1}} \smile \et_{g_{j_2}}
         \smile \cdots \smile \et_{g_{j_r}}.
\]
We call it the {\emph{$r$-th elementary symmetric polynomial of
$(\et_g)_{g \in F}$}}.
\end{Def}

Since $k$ is even and $R$ is commutative,
$\sm_r \bigl( (\et_g)_{g \in F} \bigr)$ does not
depend on the enumeration of~$F$.

\begin{Def}\label{D_0904_Def_St_mcid}
Let $X$ be a compact metrizable space,
let $G$ be a countable amenable group,
and let $\Act$ be an action of $G$ on $X$.
Let $R$ be a commutative unital ring.
For any \emph{even} integer~$k$,
we take $\mcidgs_k (\Act; R)$
to be the largest $d \in [0, \I)$
such that the following happens.

There are a finite open cover $\cU$ of~$X$
such that $\calD_X (\cU) \in \{ k, k + 1 \}$
and $\et \in \ch^k (X; \cU; R)$
(Definition~\ref{D_0814_CechFromCov}: \v{C}ech classes using $\cU$)
such that for every finite subset $G_0 \S G$ and every $\ep > 0$
there are a $(G_0, \ep)$-invariant nonempty finite set $F \S G$
and $r \in \{ 0, 1, 2, \ldots, \crd (F) \}$
for which:
\begin{enumerate}
% [label=$\mathrm{(\arabic*)}$]
%
\item\label{Item_D_0907_mcids_Prod}
Following Definition~\ref{N_0904_ElSymmP},
we have $\sm_r \bigl( (T_g^{*} (\et) )_{g \in F} \bigr) \neq 0$.
\item\label{Item_D_0907_mcids_Large}
$\dfrac{k r}{\crd (F)} > d - \ep$.
\end{enumerate}

We then say that $\Act$ has \emph{symmetric mean $k$-th
cohomological independence dimension $d$ with coefficients in $R$}.

We define the
\emph{symmetric mean cohomological independence dimension}
$\mcidgs (\Act; R)$
to be the supremum of $\mcidgs_k (\Act; R)$ over all even $k \in \N$.
\end{Def}

As we will see,
this definition is sometimes useful for subshifts of the shift on $Z^G$
when $\ch^k (Z; \Q) \neq 0$.
It doesn't give anything useful for the shift on $([0, 1]^d)^G$.
One would like to ask that
for every $\ep > 0$ there be a closed subset $Y \S X$,
a finite open cover $\cU$ of $Y$ in~$X$
such that $\calD_Y (\cU) \in \{ k, k + 1 \}$,
and $\et \in \ch^k (Y; \cU; R)$
such that for every finite subset $G_0 \S G$ and every $\dt > 0$
there are a $(G_0, \dt)$-invariant nonempty finite set $F \S G$
and $r \in \{ 0, 1, 2, \ldots, \crd (F) \}$
for which (\ref{Item_D_0907_mcids_Prod})
and~(\ref{Item_D_0907_mcids_Large}) hold.
This notion can be used on the full shift,
but does not seem useful for minimal systems.
The set $\bigcap_{g \in G} T_g^{-1} (Y)$ is closed and $G$-invariant.
If it is~$X$, then $Y = X$.
If it is~$\varnothing$,
then there is a finite subset $F \S G$
such that $\bigcap_{g \in F} T_g^{-1} (Y) = \varnothing$,
which spoils~(\ref{Item_D_0907_mcids_Large}).

We could give a generalization of this definition,
possibly useful for nonminimal systems,
by considering the supremum
of the values of $\mcidgs_k (\Act_{Y}; R)$ as $Y$ ranges over all closed
invariant subsets of $X$;
we chose to avoid it here in order to lighten notation.

\begin{Lemma}\label{R_0904_Smalller}
In the situation of Definition~\ref{D_0904_Def_St_mcid},
we have $\mcidgs_k (\Act; R) \leq \mcid_k (\Act; R)$.
\end{Lemma}

\begin{proof}
The quantity in Definition~\ref{Def_mcid_Rev}
cannot become larger when we impose the restriction $Y = X$.
Given that restriction,
the inequality to be proved follows from the fact that
if $\sm_r \bigl( (T_g^{*} (\et) )_{g \in F} \bigr) \neq 0$,
then at least one of its terms must be nonzero,
that is, there is a subset $F_0 \S F$ with $\crd (F_0) = r$
and such that the cup product of $T_g^{*} (\et)$ over all $g \in F_0$
is nonzero.
\end{proof}

In the following proposition,
the assumption that $R$ is a principal ideal
domain is needed in order to use the K{\"u}nneth Formula.

\begin{Prop}\label{Prop_Symm_mcid_subshifts}
Let $G$ be a countable amenable group.
Let $R$ be a principal ideal domain.
Let $k$ be an even integer,
and let $Z$ be a finite CW-complex with $\dim (Z) \in \{ k, k + 1 \}$
and $\ch^k (Z; R) \neq 0$.
Let $X$ be a closed $G$-invariant subset of~$Z^G$,
and let $\Act$ be the restriction to $X$
of the shift action of $G$ on~$Z^G$.
Suppose that there is an $X$-unconstrained subset of~$G$
(Definition~\ref{D_0908_Full})
with density at least~$\rh$ (Definition~\ref{D_0908_Density}).
Then $\mcidgs_k (\Act; R) \geq k \rho$.
\end{Prop}

\begin{proof}
Let $J \subset G$ be $X$-unconstrained with witness~$z$
and $\dt (J) \geq \rh$.
By translation,
we may assume without loss of generality that $1 \in J$.
If $k \rh = 0$ there is nothing to prove, so assume that $k \rh > 0$.
In particular, $k \geq 2$.

Choose an open cover $\cV$ of~$Z$
such that $\ch^k (Z; \cV; R) \neq 0$,
and let $\et_0 \in \ch^k (Z; \cV; R)$ be nonzero.
Let $q \colon Z^G \to Z$ be the projection onto the coordinate $g = 1$.
Then $q (X) = Z$.
Set $\cU = q^{-1} (\cV) \cap X$, and set $\eta = (q |_X)^* (\eta_0)$.

We claim that for any finite set $F \subset G$,
if we set $r = \crd (F \cap J)$,
then $\sm_r \bigl( (T_g^{*} (\et) )_{g \in F} \bigr) \neq 0$.
To prove the claim, first set $n = \crd (F)$.
Define maps
\[
t \colon Z^{F \cap J} \to X
\andeqn
p \colon X \to Z^{F}
\]
as follows.
Take $p$ to be the restriction of the projection map
\[
Z^{F} \times Z^{G \smallsetminus F} \to Z^{F}.
\]
For $x \in Z^{F \cap J}$
define
\[
t (x)_g = \begin{cases}
   x_g & \hspace*{1em} g \in F \cap J
        \\
   z_g & \hspace*{1em} g \in G \smallsetminus (F \cap J).
\end{cases}
\]
The fact that $z$ is a witness for the unconstrainedness of~$J$
implies that $t (x)$ as defined here really is in~$X$.

Enumerate the elements of $F$ as
$g_1, g_2, \ldots, g_n$,
with $g_1, g_2, \ldots, g_r \in F \cap J$.
The K\"{u}nneth Formula for \v{C}ech cohomology
(applied to finite CW~complexes)
gives an injective unital ring homomorphism
\[
\io \colon \bigotimes_{j = 1}^{n} \ch^{*} (Z; R) \to \ch^{*} (Z^F; R)
\]
such that, with $q_j \colon Z^F \to Z$
being the projection to the $g_j$-th coordinate,
and with $\mu_j \in \ch^{*} (Z; R)$ for $j = 1, 2, \ldots, n$,
we have
\[
\io (\mu_1 \otimes \mu_2 \otimes \cdots \otimes \mu_n)
= q_1^* (\mu_1) \smile q_2^* (\mu_2) \smile \cdots \smile q_n^* (\mu_n).
\]
Similarly,
we get an injective unital ring homomorphism
\[
\io_0 \colon
  \bigotimes_{j = 1}^{r} \ch^{*} (Z; R) \to \ch^{*} (Z^{F \cap J}; R).
\]
For $j = 1, 2, \ldots, n$,
set
\[
\ld_j
 = q_{j}^* (\et_0)
 = \io (1 \otimes 1 \otimes \cdots \otimes 1
 \otimes \et_0 \otimes 1 \otimes \cdots \otimes 1),
\]
with $\et_0$ in position~$j$.
We have
\[
\sm_r (\ld_1, \ld_2, \ldots, \ld_n)
 = \sum_{1 \leq j_1 < j_2 < \cdots < j_r \leq n}
   \ld_{j_1} \smile \ld_{j_2} \smile \cdots \smile \ld_{j_r}.
\]
Let $j \in \{ r + 1, \, r + 2, \, \ldots, n \}$
(the set indices
corresponding to the elements of $F \smallsetminus (F \cap J)$).
Then,
since $q_{j} \circ p \circ t$
is the constant map with value $z_{g_j}$,
we have
\[
(p \circ t)^* (\ld_j)
 = (q_{j} \circ p \circ t)^* (\et_0)
 = 0.
\]
Therefore
\begin{equation}\label{Eq_0907_ptstar}
\begin{split}
(p \circ t)^* \bigl( \sm_r (\ld_1, \ld_2, \ldots, \ld_n) \bigr)
& = (p \circ t)^* ( \ld_{1} \smile \ld_{2} \smile
    \cdots \smile \ld_{r})
\\
& = \io_0 (\et_0 \otimes \et_0 \otimes \cdots \otimes \et_0)
\end{split}
\end{equation}
(with $r$ tensor factors in the last expression).
Since $\et_0 \otimes \et_0 \otimes \cdots \otimes \et_0 \neq 0$,
the expression~(\ref{Eq_0907_ptstar}) is nonzero.
By naturality,
we have
\[
\sm_r \bigl( (\Act_g^* (\et))_{g \in F} \bigr)
 = p^* \bigl( \sm_r (\ld_1, \ld_2, \ldots, \ld_n) \bigr).
\]
So
$t^* \bigl( \sm_r \bigl( (\Act_g^* (\et))_{g \in F} \bigr) \bigr)
 \neq 0$,
whence $\sm_r \bigl( (T_g^{*} (\et) )_{g \in F} \bigr) \neq 0$,
as claimed.

We have $\calD_Z (\cV) \leq k + 1$ because $\dim (Z) \leq k + 1$,
and it follows that $\calD_Z (\cU) \leq k + 1$.
Since $k \rh > 0$,
we have $J \neq \varnothing$.
Choosing any finite subset $F \S G$ with $F \cap J \neq \varnothing$,
the claim certainly implies $\et \neq 0$.
Since $\et \in \ch^k (X; \cU; R)$,
this implies $\calD_Z (\cU) \geq k$.

Now let $G_0 \S G$ be finite and let $\ep > 0$.
It follows from Definition~\ref{D_0908_Density} that
there is a nonempty finite $(G_0, \ep)$-invariant subset $F \S G$
such that
\[
\frac{\crd (J \cap F)}{\crd (F)} > \rh - \frac{\ep}{k}.
\]
Then, using the claim for the second equation,
the number $r = \crd (J \cap F)$ satisfies
\[
\dfrac{k r}{\crd (F)} > k \rh - \ep
\andeqn
\sm_r \bigl( (T_g^{*} (\et) )_{g \in F} \bigr) \neq 0.
\]
This completes the proof.
\end{proof}

\begin{Thm}\label{T_0907_gs}
Let $G$ be a countable amenable group,
let $X$ be a compact metrizable space,
and let $\Act$ be an action of $G$ on $X$.
Let $k$ be an even integer.
Let $m$ be the greatest integer
with $m < \frac{1}{2} \mcidgs_k (\Act; \Q)$.
Then $\rc (C (X) \rtimes_{\Act} G) \geq m$.
If $C (X) \rtimes_{\Act} G$ is simple,
then $\rc (C (X) \rtimes_{\Act} G) > m$.
\end{Thm}

In particular,
$\rc (C (X) \rtimes_{\Act} G)
 \geq \frac{1}{2} \mcidgs_k (\Act; \Q) - 1$.

\begin{proof}[Proof of Theorem~\ref{T_0907_gs}]
Fix an even integer~$k$.
Fix $\eps_0 > 0$.

Choose a finite open cover $\cU$ of $X$ such that
$\calD_X (\cU) \in \{ k, \, k + 1 \}$
and a cohomology class $\eta \in \ch^k (X; \cU; \Q)$ such that
for any $\dt > 0$
and any finite subset $G_0 \subset G$
there exist a nonempty finite $(G_0, \delta)$-invariant subset $F$
and $r \in \{ 0, 1, 2, \ldots, \crd (F) \}$ satisfying:
\begin{enumerate}
\item\label{T_0907_gs_Nonzero_New}
Following Definition~\ref{N_0904_ElSymmP},
we have $\sm_r \bigl( (T_g^{*} (\et) )_{g \in F} \bigr) \neq 0$.
\item\label{T_0907_gs_Folner_New}
$\dfrac{k r}{\crd (F)} > \mcidgs_k (\Act; \Q) - \dt$.
\end{enumerate}
%
\begin{comment}
Choose a finite open cover $\cV$ of $X$
such that $\cV \prec \cU$ and $\ord (\cV) = \calD_X (\cU)$.
Recall that $\cN (\cV)$ is the nerve of~$\cV$.
By the construction of \v{C}ech cohomology,
$\eta$ is also in $\ch^k (X; \cV; \Q)$,
and can be obtained as a pullback of a cohomology class
$\eta' \in \ch^k (\cN (\cV); \Q)$
by a map $f \colon X \to \cN (\cV)$.
Use Lemma \ref{Lemma:zero lower classes}
to choose a vector bundle $E'$ over $\cN (\cV)$
such that $c_k^{\Q} (E') = M \eta'$ for some nonzero integer~$M$
and $c_j^{\Q} (E') = 0$
for all $j \in \{ 1, 2, \ldots, k / 2 - 1 \}$.
Since $\dim ( \cN ( \cV)) < k + 2$,
we also have $c_j^{\Q} (E') = 0$
for all $j \in \{ k / 2 + 1, \, k / 2 + 2, \ldots \}$.
Therefore $c^{\Q} (E) = 1 + M \et'$.
As we can replace $\eta$ by any nonzero scalar multiple
of it and retain the same properties,
we may use $M\eta$ in place of $\eta$,
so we may simplify notation and assume that $c (E') = 1+\eta'$.
Since $\dim (\cN (\cV)) \leq k+1$,
by subtracting trivial bundles
(see Theorem 9.1.2 of~\cite{Husemoller}),
we can also assume that $\rank (E') = k/2$.
Now set $E = f^* (E')$.
We have $c (E) = 1+\eta$ by naturality, and $\rank (E) = k/2$.
\end{comment}
Arguing as in the proof of Theorem \ref{T_0826_Main},
after possibly replacing
$\eta$ by a nonzero scalar multiple of itself,
we can choose a vector bundle
$E$ over $X$ such that $\rank (E) = k/2$ and such that $c (E) = 1+\eta$.

Let $L$ be the dimension of some trivial bundle
which has $E$ as a direct summand, and
let $a \in M_L (C (X))$ be the projection onto $E$.
Let $b \in M_{\infty} (C (X))$ be a constant projection.
Suppose that $a \precsim_{C (X) \rtimes_{\Act} G} b$.
We are going to prove that
this implies
$\rank (b) \geq \frac{1}{2} ( \mcidgs_k (\Act; \Q) + k) - \ep_0$.

\begin{comment}
Increasing $L$ if needed,
we may assume that $b \in M_{L} (C (X))$.
There is $c_0 \in M_{\infty} (C (X) \rtimes_{\Act} G)$
such that $\| c_0^* b c_0 - a \| <  1/8$.
Replacing $c_0$ by its cutdown by
$1_{M_{L} (C (X) \rtimes_{\Act} G)}$,
we may assume that
\[
c_0
 \in M_{L} (C (X) \rtimes_{\Act} G)
 \cong (M_{L} \otimes C (X)) \rtimes_{\id_{M_{L}} \otimes \Act} G.
\]
\end{comment}
Arguing again as in the proof of Theorem \ref{T_0826_Main}
(in particular, possibly increasing~$L$),
we may assume that $b \in M_{L} (C (X))$
and that there is $c \in M_{L} (C (X) \rtimes_{\Act} G)$
in the algebraic crossed product
such that
\begin{equation}\label{Eq_0907_Symm_cbc}
\| c^* b c - a \| <  \frac{1}{4}.
\end{equation}
%
% Being in the algebraic crossed product means that,
With $u_g \in C (X) \rtimes_{\Act} G$ being the standard
unitary corresponding to $g \in G$,
there are a finite subset $G_0 \subset G$
and elements $c_g \in M_L (C (X))$ for $g \in G_0$
such that $c = \sum_{g \in G_0} c_g (1_{M_L} \otimes u_g)$.
%  Increasing the size of the set~$G_0$,
%  we may assume that
%  %
%  \begin{equation}\label{Eq_0907_Symm_Inverse}
%  G_0 = \bigl\{ g^{-1} \mid g \in G_0 \bigr\}.
%  \end{equation}
%  %
Choose $\delta  > 0$ such that
\[
\delta \cdot (\| c \| + 1) \cdot \sum_{g \in G_0} \|c_g\|
 < \frac{1}{16}.
\]

We proceed to construct a cutoff function,
where here we need to make choices
which are a bit different than those used in the proof of Theorem
\ref{T_0826_Main}.
Choose a nonempty finite $(G_0, \delta)$-invariant subset $G_1$ of $G$
with $1 \in G_1$.
Choose $\eps_1 > 0$ such that
\begin{equation}\label{0907_Symm_CardEst}
\frac{ \mcidgs_k (\Act; \Q) + k - \eps_0}{1 + \eps_1}
 > \mcidgs_k (\Act; \Q) + k - 2 \eps_0.
\end{equation}
Set
\[
\ep_2 = \min \left( \ep_0, \, \frac{\ep_1}{2 \crd (G_1)^2} \right).
\]
Find a finite nonempty
$\bigl( G_0 \cup G_1, \, \eps_2 \bigr)$-invariant
subset $F \subset G$
and $r \in \{ 0, 1, 2, \ldots, \crd (F) \}$ satisfying
%  %
%  \begin{enumerate}
%  %
%  \item\label{Eq_0909_Prod}
%  $\sm_r \bigl( (T_g^{*} (\et) )_{g \in F} \bigr) \neq 0$.
%  %
%  \item\label{Eq_0909_Folner}
%  $\dfrac{k r}{\crd (F)} > \mcidgs_k (\Act; \Q) - \ep_0$.
%  \end{enumerate}
%  %
%
\begin{equation}\label{Eq_0908_MatchDfn}
\sm_r \bigl( (T_g^{*} (\et) )_{g \in F} \bigr) \neq 0
\andeqn
\dfrac{k r}{\crd (F)} > \mcidgs_k (\Act; \Q) - \ep_0.
\end{equation}

Set
\[
S = \bigcup_{h \in G_1} h^{-1} F.
\]
Then
\begin{equation}\label{Eq_0907_Symm_FppEst}
\begin{split}
\crd (S \smallsetminus F)
& \leq \sum_{h \in G_1} \crd ( h^{-1} F \smallsetminus F)
  = \sum_{h \in G_1} \crd ( F \smallsetminus h F)
\\
& < \crd (G_1) \ep_2 \crd (F)
  \leq \ep_1 \crd (F).
\end{split}
\end{equation}

Let $\Delta \colon G \to [0, 1]$ be the function
\[
\Delta
 = \frac{1}{\crd (G_1)} (\chi_{G_1} * \chi_S).
\]
% for $g \in G$.
(We have replaced $F$ in the proof of Theorem~\ref{T_0826_Main}
with~$S$.)
By the same reasoning as there,
for $t \in G_0$ and $g \in G$ we have
\begin{equation}
\label{eqn_Delta_almost_invariant_Symm}
| \Delta (t^{-1} g) - \Delta (g) |
  \leq \frac{\crd \bigl( (t G_1 \smallsetminus G_1)
      \cup (G_1 \smallsetminus t G_1) \bigr)}{\crd (G_1)}
  < 2 \delta.
\end{equation}
Likewise, by similar reasoning to that
in the proof of Theorem~\ref{T_0826_Main},
we see that that $\Delta (g) = 1$ for all $g \in F$.
%To see this, note that
%
% \begin{equation}\label{Eq_0907_Symm_Dt}
%\[
%\Delta(g)
%  = \frac{1}{\crd (G_1)}
%    \sum_{h \in G} \chi_{G_1} (h) \chi_S (h^{-1} g)
%  = \frac{1}{\crd (G_1)} \sum_{h \in G_1} \chi_{S} (h^{-1} g).
%\]
% \end{equation}
%
%So $\Delta (g) = 1$
%if and only if $h^{-1} g \in S$ for all $h \in G_1$.
%This certainly holds for all $g \in F$,
%implying the claim.

As in the proof of Theorem \ref{T_0826_Main},
let $\af$ be the corresponding action of $G$ on $C (X)$,
and view $C (X) \rtimes_{\Act}G $
as embedded in $B (l^2 (G) \otimes C (X))$ in the same way as there.
\begin{comment}
(that is, the adjointable operators on the Hilbert $C (X)$-module
$l^2 (G) \otimes C (X)$)
in the standard way, that is,
if $f \in C (X)$, $\xi \in l^2 (G) \otimes C (X)$,
and $g, h \in G$, then
\[
(f \cdot \xi) (g) = \alpha_{g^{-1}} (f) \cdot \xi (g)
\andeqn
(u_h \cdot \xi) (g) = \xi (h^{-1} g).
\]
\end{comment}
We define a multiplication operator
$d_0 \in B (l^2 (G) \otimes C (X))$ by
\[
(d_0 \xi) (g) = \Delta (g) \xi (g)
\]
for $g \in G$.
By~(\ref{eqn_Delta_almost_invariant_Symm}),
for any $t \in G_0$,
$\| u_t d_0 u_t^* - d_0 \| < 2 \delta$,
whence $\| u_t d_0 - d_0 u_t \| < 2 \delta$.
Set $d = d_0 \otimes 1_{M_L}$.
It follows that
\begin{equation}\label{Eq_0907_Symm_CommEst}
\| c d - d c \|
 < 2 \delta \sum_{g \in G_0} \|c_g\|
 < \frac{1}{8 (\| c \| + 1)}.
\end{equation}

Notice that $\supp (\Delta) = G_1 S$.
Thus, we can view $d  M_{L} (C (X) \rtimes_{\Act} G) d$
as included in $B \bigl( l^2 (G_1 S) \otimes M_L \otimes C (X) \bigr)$,
with $l^2 (G_1 S) \otimes M_L \otimes C (X)$
regarded as a Hilbert $M_L \otimes C (X)$-module.
Since $G_1 S$ is a finite set,
this is a matrix algebra over $C (X)$.

Since, in particular,
$F$ is $( G_1, \eps_2 )$-invariant,
and since
\[
G_1 S \smallsetminus F
 = \bigcup_{g, h \in G_1} (g h^{-1} F \smallsetminus F)
 \S \bigcup_{g, h \in G_1}
   \bigl[ g (h^{-1} F \smallsetminus F)
             \cup (g F \smallsetminus F) \bigr],
\]
we have
\begin{equation}\label{Eq_0907_Symm_FpF}
\crd (G_1 S \smallsetminus F)
 \leq 2 \crd (G_1)^2 \ep_2 \crd (F)
 < \eps_1 \crd (F).
\end{equation}
Set
\[
c' = d^{1/2} c d^{1/2}
   \in B \bigl( l^2 (G_1 F) \otimes M_L \otimes C (X) \bigr).
\]
Then, arguing as in the proof of Theorem~\ref{T_0826_Main}, we have
\[
\bigl\| (c')^* d^{1/2} b d^{1/2} c'
- d^{3/2} a d^{3/2} \bigr\| < \frac{1}{2}.
\]
%at the third step
%using (\ref{Eq_0907_Symm_CommEst}) on the first term
%and (\ref{Eq_0907_Symm_cbc}) on the second term,
%
%\begin{equation*}
%\begin{split}
%& \bigl\| (c')^* d^{1/2} b d^{1/2} c'
%               - d^{3/2} a d^{3/2} \bigr\|
%\\
%& \hspace*{3em} {\mbox{}}
% \leq
% \bigl\| d^{1/2} c^* d b d c d^{1/2}
%                - d^{3/2} c^* b d c d^{1/2} \bigr\|
%   + \bigl\| d^{3/2} c^* b d c d^{1/2}
%                - d^{3/2} c^* b c d^{3/2} \bigr\|
%   \\
%& \hspace*{6em} {\mbox{}}
%     + \bigl\| d^{3/2} c^* b c d^{3/2}
%                - d^{3/2} a d^{3/2} \bigr\|
%\\
%& \hspace*{3em} {\mbox{}}
% \leq 2 \| c \| \| d c - c d \| + \| c^* b c - a \|
%  < 2 \left( \frac{1}{8} \right) + \frac{1}{4}
%  = \frac{1}{2}.
%\end{split}
%\end{equation*}
%
Under our identification of $a$
as an element in the crossed product,
$d^{3/2} a d^{3/2}$ is a diagonal operator
on the Hilbert $M_L \otimes C (X)$-module
$l^2 (G_1 F) \otimes M_L \otimes C (X)$,
and for any $g \in F$, since $\Delta (g) = 1$,
the $g$-th diagonal entry
is simply $(\alpha_{g^{-1}} \otimes \id_{M_L}) (a)$.
Likewise, as $b$ is invariant under the group action,
$d^{3/2}b d^{3/2}$ is a diagonal matrix
whose diagonal entries
are scalar multiples of the constant projection $b$.

Let $p \in B \bigl( l^2 (G_1 S) \otimes M_L \otimes C (X) \bigr)$
be the diagonal projection
whose diagonal $g$-th entry is $1$ if $g \in F$ and zero otherwise.
We obtain
\[
\bigl\| p d^{3/2} a d^{3/2} p
        - p (c')^* d^{1/2} b d^{1/2} c' p \bigr\|
 < \frac{1}{2}.
\]
Note that $p d^{3/2} a d^{3/2} p = p a p$.
Thus, the projection $p a p$
is Murray-von-Neumann subequivalent
to the cutdown of $b$ to
$B \bigl( l^2 (G_1 S) \otimes M_L \otimes C (X) \bigr)$,
which is a constant projection of rank $\crd (G_1 S)\cdot \rank (b)$.
Notice that
$p a p$ is the projection onto
${\widetilde{E}} = \bigoplus_{g \in F} \Act_g^* (E)$.
Now,
\[
c \bigl( {\widetilde{E}} \bigr)
 = {\displaystyle{ \underset{g \in F}{\smile}
       (1 + \Act_g^* (\eta)) }},
\]
so if ${\widetilde{E}} \oplus E'$ is a trivial bundle, then
\[
c (E')
 = c \bigl( {\widetilde{E}} \bigr)^{-1}
 = {\displaystyle{
    \underset{g \in F}{\smile}
       (1 - \Act_g^* (\eta)) }}
 = \sum_{j = 0}^{\crd (F)}
    (-1)^j \sm_j \bigl( (T_g^* (\et))_{g \in F} \bigr).
\]
In particular,
\[
c_{k r / 2} (E')
 = (-1)^r \sm_r \bigl( (T_g^* (\et))_{g \in F} \bigr)
 \neq 0,
\]
so $\rank (E') \geq k r / 2$.
Therefore ${\widetilde{E}}$ does not embed
in a trivial bundle of rank less than
\[
\frac{k r}{2} + \rank \bigl( {\widetilde{E}} \bigr)
 = \frac{k}{2} ( r + \crd (F) ).
\]
So $\crd (G_1 S) \cdot \rank (b) \geq \frac{k}{2} ( r + \crd (F) )$,
whence, using (\ref{Eq_0907_Symm_FpF}) at the second step,
(\ref{Eq_0908_MatchDfn}) at the fourth step,
and (\ref{0907_Symm_CardEst}) at the fifth step,
\[
\begin{split}
\rank (b)
& \geq \frac{k}{2} \left( \frac{r + \crd (F)}{\crd (G_1 S)} \right)
  > \frac{k}{2}
     \left( \frac{r + \crd (F)}{(1 + \ep_1) \crd (F)} \right)
\\
& = \frac{1}{2}
      \left( \frac{\frac{k r}{\crd (F)} + k}{1 + \ep_1} \right)
  > \frac{1}{2}
     \left( \frac{\mcidgs_k (\Act; \Q) - \eps_0 + k}{1 + \eps_1} \right)
\\
& > \frac{1}{2} \bigl( \mcidgs_k (\Act; \Q) + k \bigr) - \eps_0,
\end{split}
\]
as wanted.

Now recall that $m$ is the greatest integer
with $m < \frac{1}{2} \mcidgs_k (\Act; \Q)$.
Choose $\ep_0 > 0$ such that
$\frac{1}{2} \mcidgs_k (\Act; \Q) - \ep_0 > m$.
In the argument above,
use this value of $\ep_0$,
and take $b$ to be a constant projection with
$\rank (b) = m + \frac{k}{2}$.
Then $a \not\precsim_{C (X) \rtimes_{\Act} G} b$.
For any invariant tracial state $\tau$ on $C (X)$,
and hence for any tracial state $\tau$ on $C (X) \rtimes_{\Act} G$,
we have $d_{\tau}(b) = \rank (b)$
and $d_{\tau}(a) = \rank (a) = k/2$.
So $C (X) \rtimes_{\Act} G$ does not
have $m$-comparison.
Thus, $\rc (C (X) \rtimes_{\Act} G) \geq m$.

The argument for $\rc (C (X) \rtimes_{\Act} G) > m$
when $C (X) \rtimes_{\Act} G$ is simple
is the same as in the proof of Theorem~\ref{T_0826_Main}.
\end{proof}

\begin{Cor}\label{C_0907_Shift_Sk}
Let $k$ be a strictly positive even integer,
let $Z$ be a $k$-dimensional polyhedron
such that $\ch^k (Z; \Q) \neq 0$,
let $G$ be a countable amenable group,
and let $\rh \in (0, 1)$.
Let $(X, \Act)$ be the minimal subshift of the shift on $Z^G$
constructed in Section~4 of~\cite{Dou2017}
to satisfy $\mdim (\Act) = k \rh$.
Then $\rc (C (X) \rtimes_{\Act} G) > \frac{1}{2} \mdim (T) - 1$.
\end{Cor}

\begin{proof}
Let $J \subset G$ be as in Proposition~\ref{P_0908_ExistFull}.
This proposition implies that
the hypotheses of Proposition~\ref{Prop_Symm_mcid_subshifts}
are satisfied,
so $\mcidgs_k (\Act; \Q) \geq k \rh$.
Theorem~\ref{T_0907_gs} now gives
$\rc (C (X) \rtimes_{\Act} G) > \frac{1}{2} k \rh - 1$.
As in~\cite{Dou2017},
we have $\mdim (\Act) = k \rh$.
\end{proof}

In particular,
if $k$ is even, then the subshifts of $(G, (S^k)^G )$
in \cite{Dou2017} satisfy
$\rc (C (X) \rtimes_{\Act} G) > \frac{1}{2} \mdim (T) - 1$.
This is within~$1$ of the conjectured value
of $\rc (C (X) \rtimes_{\Act} G)$.

\section{Concluding remarks}\label{Sec_0922_CR}

For minimal subshifts as in Section \ref{Sec_0908_Symm},
we can bound the
radius of comparison of the crossed product
$\rc (C (X) \rtimes_{\Act} G)$ from
below by the largest integer smaller than $k\rho/2$,
where $\rho$ is the density as in that example.
For suitable choices of $\rho$, this can be arbitrarily
close to $k\rho/2$.
The fact that we only get integers is a
consequence of our method of proof.
We know of no reason not to believe that the following
may have an affirmative answer.

\begin{Question}
Let $T$ be a topologically free and minimal action
of a countable amenable
group $G$ on a compact metrizable space $X$.
Do we have
\[
\frac{1}{2}\mdim(\Act) \geq \rc (C (X) \rtimes_{\Act} G)
 \geq \frac{1}{2}\mcid(\Act ; \Q) \; ?
\]
\end{Question}

For the lower bound, minimality does not seem to be relevant.

In Proposition~\ref{Prop_mcid_for_subshifts},
we proved that if $(X, G)$ is a subshift of $Z^G$
which has an $X$-unconstrained set $J \S G$ with density~$\rh$,
then
\[
\mcid (\Act |_X; R)
 \geq 2 \rho \left\lfloor \frac{\dim(Z) - 1}{2} \right\rfloor.
\]
Proposition 2.8 of \cite{Krieger2009}
gives a related estimate with mean dimension
in place of mean cohomological independence dimension.
While this does not show that they coincide,
it does mean that for reasonable spaces they are not far apart,
and suggests that they may coincide under reasonable conditions.

\begin{Question}\label{Q_0826_Coeffs}
Does mean cohomological independence dimension
coincide with mean dimension for subshifts
under the hypotheses of Proposition \ref{Prop_mcid_for_subshifts}?
Does this depend on the ring of coefficients?
\end{Question}

The space $Z$ in Proposition \ref{Prop_mcid_for_subshifts}
is a finite CW-complex.
Bad spaces may well behave quite differently.
Scattered in~\cite{DranishnikovSurvey},
one can find various examples of strange behavior
of cohomological dimension in products,
different for different coefficient rings,
and differences between cohomological dimension and covering dimension.
We don't know the mean cohomological independence dimension
of shifts or subshifts on badly behaved spaces.
(For the mean dimension of shifts
on finite dimensional badly behaved spaces, see~\cite{tsukamuto}.
That paper leaves open the mean dimension of the shift on,
for example, a compact space $X$ with $\dim (X) = \I$
but integer cohomological dimension $\dim_{\Z} (X) = 3$.
Shifts on such spaces are perhaps more likely to exhibit
strange behavior of mean cohomological independence dimension.)
%We also don't know the mean dimension;
%for example, for $G = \Z$,
%Proposition~3.3 of~\cite{LinWeiss2000}
%only helps when $Z$ contains a cube of dimension $\dim (Z)$,
%which is not the case for these examples.

\end{document}